\documentclass[A4paper,12pt]{article}
\usepackage[utf8]{inputenc}
\usepackage{tikz}

\usepackage{amsmath,amsthm,amscd,amssymb,eucal,mathrsfs}
\usepackage{amsfonts}
\usepackage{url}
\usepackage{latexsym}
\usepackage{graphicx} 
\usepackage{fullpage}
\usepackage{multicol}
\usepackage{dsfont}
\usepackage{natbib}
\usepackage[all]{xy}
\usepackage{multirow}
\usepackage{color}
\usepackage{hyperref}

\newtheorem{thm}{Theorem}[section]

\newtheorem{cor}[thm]{Corollary}
\newtheorem{remark}[thm]{Remark}
\newtheorem{definition}[thm]{Definition}

\newtheorem{Proposition}[thm]{Proposition}

\newtheorem*{Satz*}{Satz}

\newtheorem{Lemma}[thm]{Lemma}

\newcommand{\mathset}[1]{{\left\{#1\right\}}}
\newcommand{\absolute}[1]{\left\lvert#1\right\rvert}
\newcommand{\norm}[1]{\left\|#1\right\|}

\newcommand{\modulus}{{\rm\;\,\,\!\!\! mod}}

\DeclareMathOperator{\closure}{cl}
\DeclareMathOperator{\Spec}{Spec}

\DeclareMathOperator{\dom}{dom}

\DeclareMathOperator{\diam}{diam}
\DeclareMathOperator{\supp}{supp}

\DeclareMathOperator{\can}{can}
\DeclareMathOperator{\Hom}{Hom}
\DeclareMathOperator{\divisor}{div}

\title{Searching Schemes With $p$-Adic Neumann Boundary Value Problems}
\author{Patrick Erik Bradley
\\
Karlsruhe Institute of Technology
\\
Geodetic Institute
\\
Englerstr.\ 7
\\
76131 Karlsruhe
\\
Germany
\\
Email: bradley@kit.edu
}
\date{\today}

\begin{document}

\maketitle

\begin{abstract}
Firstly, for branched covering maps $f\colon Y\to X$ between smooth, separated schemes  of finite type over the  integral ring $O_K$ of a non-archimedean local field $K$, the ramification divisor is found to coincide with the divisor of the Radon-Nikodym derivative of the Radon measure associated with an algebraic differential form on $Y$ against the pullback measure of one on $X$, both taking Borel sets of the space $Y(O_K)$ of $O_K$-rational points of $Y$ as input values. 
Secondly, the ultracontractivity of semigroups generated by Laplacian integral operators having a geodetic-distance-power kernel function is proven, using heat kernel estimates for ultrametric operators on $p$-adic balls. Thirdly, 
a series of $p$-adic Neumann Boundary Value Problems, depending on algebraic and pluricanonical differential forms with poles on spaces $X(O_K)$ coming from schemes is formulated and solved, extending previous work of the author. These results are then used to solve reconstruction problems on schemes: the divisor of a pluricanonical form with poles with at worst log-terminal singularities, as well as Weierstrass points of projective algebraic curves can be reconstructed via repeatedly finding weak solutions of $p$-adic Neumann Boundary Value Problems. 
\end{abstract}

\section{Introduction}

The use of $p$-adic integration in arithmetic geometry has been found to be helpful in studying $p$-adic Hodge theory \cite{Ito2004,HR2008}, the wild McKay correspondence \cite{Yasuda2017}, or mirror symmetry \cite{GWZ2020}, just in order to name a few recent instances of such applications, likely inspired by \cite{Batyrev1999}. Underlying such a method is an integral structure on the scheme at hand, and then  measures induced by naturally obtained top differential forms are obtainable, as explained e.g.\ in the recent article \cite{BKL2026}, and masterfully exploited towards arithmetic applications, as in \cite{Oesterle1984}.
\newline

This insight opens the door to applications of $p$-adic analysis on $p$-adic analytic manifolds, as has been in the recent focus of the author's research \cite{diffMfp,HearingSerre}. It belongs to the ongoing project \cite{brad_habil}, from which the very recent formulation and study of $p$-adic Neumann Boundary Value Problems in \cite{brad_nbvp} takes upon itself applications aiming at extracting information from $p$-adic analytic manifolds, in particular when these are given as the set of integral points of a scheme defined over the integral ring of a non-archimedean local field $K$. The case of branched coverings of projective algebraic curves over $K$ studied in the end of this article forms a bridge to early work on such coverings between Mumford curves \cite{RETMumf,ExpliCycMumf}.
\newline

What this article aims at, is to widen the path connecting arithmetic geometry with
the analysis of $p$-adic stochastic processes and diffusion equations, cf.\ e.g.\ \cite{Zuniga2020,PW2025,BW2019,Weisbart2024}, also by going deeper into all mathematical disciplines participating in this connection. This is effected here by extending the scope of $p$-adic Neumann Boundary Value Problems, as found e.g.\ in \cite{AK2026,HJ2026,brad_nbvp}, to measures obtained by pluricanonical differential forms on a scheme on the one hand, and by involving pullbacks of differential forms via branched covering maps. In this way, methods from algebraic geometry become usable in translating pullback forms into the language of measure and integration. The most visible such example is the Radon-Nikodym Theorem which guarantees the existence of a distribution function $\frac{d\nu}{d\mu}$ in order to express one measure $\mu$ with another measure $\nu$, as long as they are compatible with another. This compatibility is called \emph{strongly continuous}, and amounts to $\mu$-zero sets also being $\nu$ zero sets, and then the Radon-Nikodym derivative $\frac{d\nu}{d\mu}$ exists. This is played out when a measure $\modulus(\omega_{\mathfrak{Y}})$  on a scheme $\mathfrak{Y}$, induced by an algebraic differential form $\omega_{\mathfrak{Y}}$, is expressible via the pullback $f^{*}\omega_{\mathfrak{X}}$ of a differential form $\omega_{\mathfrak{X}}$ on the the target scheme of a branched covering morphism $f\colon\mathfrak{Y}\to\mathfrak{X}$, times a function having the same zeros, including their multiplicities, as the ramification divisor $R$ of the covering map $f$. So, whereas in the algebraic description, the extra summand $R$ in the isomorphism
\[
K_{\mathfrak{Y}}\cong f^*K_{\mathfrak{X}}+R
\]
connecting the canonical divisor $K_{\mathfrak{Y}}$ with the pullback $f^*K_{\mathfrak{X}}$, is not the divisor of a function on $\mathfrak{Y}$, this is the case in the measure-theoretic description through the Radon-Nikodym derivative
\[
\frac{d\modulus(\omega_{\mathfrak{Y}})}{d\modulus(f^*\omega_{\mathfrak{X}})}
\]
having its divisor equal to $R$.
This is part of  the first main result:
\newline

\noindent
{\bf Theorem.}
The following statements holds true for the morphism $f\colon\mathfrak{Y}\to\mathfrak{X}$: 
\begin{enumerate}
\item The Radon-Nikodym derivative takes the value
\[
\frac{d\modulus(\omega_{\mathfrak{Y}})}
{d\modulus(f^*\omega_{\mathfrak{X}})}
(y)
=
\frac{\norm{\omega_{\mathfrak{Y}}}}
{\norm{f^*\omega_{\mathfrak{X}}}}
\]
for $\nu_{\mathfrak{Y},\can}$-almost all $y\in\mathfrak{Y}(O_K)$, where $\nu_{\mathfrak{Y},\can}$ is the canonical measure on $\mathfrak{Y}(O_K)$.
\item There is an equality
\[
\divisor\left(\frac{d\modulus(\omega_{\mathfrak{Y}})}{d\modulus(f^*\omega_{\mathfrak{X}})}\right)
= R
\]
of divisors on $\mathfrak{Y}(O_K)$.
\item The map
\[
\mathfrak{Y}(O_K) \to\mathfrak{R},\;
y\mapsto \frac{d\modulus(\omega_{\mathfrak{Y}})}{d\modulus(f^*\omega_{\mathfrak{X}})}
\]
is regular, i.e.\ does not contain any poles.
\end{enumerate}

The next type of results are more general formulations of Neumann Boundary Value Problems (NBVP) than in \cite{brad_nbvp}. These involve a Dirichlet form associated with a Laplacian integral operator
\[
\Delta_\Omega u(x)=\int_{\mathfrak{X}(O_K)}w_\Omega(x,y) (u(x)-u(y))\,d\modulus(\omega)(y)
\]
whose kernel function $w(x,y)$ yields a graphon structure first on $\mathfrak{X}(O_K)$, and then for an open subset  $\Omega\subset \mathfrak{X}(O_K)$ a graphon-theoretic boundary
\[
\delta_w\Omega=\mathset{y\in\mathfrak{X}(O_K)\setminus\Omega\mid w(x,y)\neq 0}
\]
and closure
\[
\closure_w\Omega=\Omega\cup\delta_w\Omega\,,
\]
together with a restricted kernel function
\[
w_\Omega(x,y)=\begin{cases}
w(x,y),&x,y\in\closure_w\Omega
\\
0,&\text{otherwise.}
\end{cases}
\]
The NBVP formulates itself as
\[
\Delta_\Omega u|_\Omega=0,\quad N_{\delta_w\Omega}u|_{\delta_w\Omega}=\phi
\]
for a given function $\phi\in L^\infty(\closure_w\Omega,\modulus(\omega))$, where 
\[
N_{\delta_w\Omega}u(x)=\int_\Omega w(x,y)(u(x)-u(y))\,d\modulus(\omega)(y)
\]
acts as the normal derivative for $x\in\delta_w\Omega$. During the course of the text, the measure $\modulus(\omega)$ will be replaced by measures associated with other differential forms, depending on the context. The theorems are existence and uniqueness results for weak solutions of the NBVP, extending the scope of \cite[Theorem 5.6]{brad_nbvp}.
\newline

The final main result is a reconstruction theorem for a pluricanonical $\mathds{Q}$-divisor
\[
\divisor(\omega)=\sum\limits_{i=1}^s a_i\mathfrak{D}_i
\]
on a scheme $\mathfrak{X}$ with poles, as long as the singularities are at worst log-terminal, via solving series of suitable NBVPs weakly. This is done by detecting the volumes of the fibres of the reduction map $\rho\colon\mathfrak{X}(O_K)\to\mathfrak{X}(\mathds{F}_q)$, the orders $a_1,\dots,a_s\in\mathds{Q}$, and also the divisor $\mathfrak{D}_j$ containing a given ball. All of this via weakly solved NBVPs. This result is followed by extensions of the Closeness-to-Zero Theorem \cite[Theorem 6.3]{brad_nbvp} in order to detect Weierstrass points on a projective algebraic curve, and then the Weierstrass points of a hyperelliptic curve via its $2$-sheeted branched covering coming from the hyperelliptic involution. Again through weakly solved suitable NBVPs.
\newline

The following Section 2 introduces necessary preliminaries not already mentioned here. This is followed by Section 3 which provides an overview on canonical measures and differential forms, including the results on the Radon-Nikodym derivative. Section 4 extends this to finite, separable, dominant maps between separated, smooth $O_K$-schemes of finite type. Section 5 is devoted to the NBVP using algebraic differential forms. Section 6 concludes with the final results stated in the previous paragraph.
\section{Preliminaries}

Analytic manifolds are defined over a non-archimedean local field $K$ just like they are defined  over the real number field $\mathds{R}$, thereby taking into account that the transition functions between overlapping charts are meant to be $K$-analytic, cf.\ \cite{Igusa2001,Schneider2011,Serre1992,WeilAAG}. 
In order to be able to construct integral Laplacian operators over their real- and complex-
valued function spaces, it was found useful to use a connected nerve complex coming from
a suitable atlas, which in the compact case can be assumed finite \cite{diffMfp}. In order to be able to
define an analogue of geodetic distance (called $p$-adic geodetic distance), a Radon measure
given by a nowhere vanishing analytic differential $n$-form on the $p$-adic analytic $n$-manifold $X$ can be used. In the case that $X$ is compact, then it is known that such
a differential $n$-form exists, cf.\ \cite[Théorème (2)]{Serre1965}.
What is to be kept in mind about the geodetic distance in the $p$-adic setting, is that it is locally $p$-adic distance on charts, given  by the maximum norm on $K^n$. If the $p$-adic analytic manifold  has an integral structure, as introduced
e.g.\ in \cite{BKL2026}, then the transition maps between overlapping charts take balls to balls of equal
radius. And that is very useful for defining the $p$-adic geodetic distance, whereas in \cite{diffMfp}, this was
done in the compact case with the so-called \emph{equalising} property of transition maps. Integral
structures for this task are used in \cite{brad_nbvp}.
In the case of an $n$-dimensional scheme $\mathfrak{X}$ defined over the ring $O_K$ of integers of a non-archimedean local field $K$, with the property of being separated, smooth, and of finite type, there is a natural structure of a compact $K$-analytic manifold on the set $\mathfrak{X}(O_K)$ of $O_K$-rational points of $\mathfrak{X}$. Furthermore, there is a canonical measure $\nu_{\mathfrak{X},\can}$ coming from a nowhere vanishing differential $n$-form, and this Radon measure can be used, together with a finite open covering of $\mathfrak{X}(O_K)$ having a connected nerve complex, to use the geodetic distance in order to define distance-based kernel functions for $p$-adic Laplacian integral operators on spaces of functions $\mathfrak{X}(O_K)\to\mathds{C}$. Notice that $\mathfrak{X}$ being an $O_K$-scheme provides $\mathfrak{X}(O_K)$ with a natural $O_K$-structure, from which most of the notions can be derived naturally. The only exception is the atlas giving rise to a connected nerve complex. Its vertices are the open sets from the finite covering $\mathcal{U}$ provided for by the atlas, and all other $k$-simplices are defined via intersections of $k$ overlapping sets in $\mathcal{U}$.  Paths in the nerve complex are defined in the usual way, as it is a simplicial complex, and this allows to extend the notion of path to all of $\mathfrak{X}(O_K)$, as each face $\sigma$ in the nerve complex has an underlying open subset $U_\sigma$ of $\mathfrak{X}(O_K)$ covered with sufficiently small $p$-adic balls, all contained in the set $U_\sigma$, now acting as the root of the tree of balls contained in $U_\sigma$. The result is an infinite partially ordered set of balls and faces, at whose boundary lies $\mathfrak{X}(O_K)$. The article \cite{diffMfp} has a more detailed description of this concept.
\newline

Given smooth morphisms of schemes:
\begin{align}\label{composedMorphism}
\xymatrix{
\mathfrak{Y}\ar[r]^\pi&\mathfrak{X}\ar[r]^\rho&\mathfrak{S}
}
\end{align}
there is an exact sequence of quasicoherent sheaves
\begin{align}\label{exactSequence_QCSh}
\xymatrix{
0\ar[r]&\pi^*\Omega_{\mathfrak{X}/\mathfrak{S}}
\ar[r]&\Omega_{\mathfrak{Y}/\mathfrak{S}}\ar[r]&\Omega_{\mathfrak{Y}/\mathfrak{X}}\ar[r]&0
}
\end{align}
cf.\ \cite[Theorem 22.2.25]{Vakil2017} together with \cite[Exercise 21.2.S]{Vakil2017}.
Since determinantal line bundles behave well in exact sequences, cf.\ 
\cite[Exercise 13.5.H]{Vakil2017}, it now follows that
\[
\omega_{\mathfrak{Y}/\mathfrak{S}}\cong
\pi^*\omega_{\mathfrak{X}/\mathfrak{S}}\otimes\omega_{\mathfrak{Y}/\mathfrak{X}}^\vee\,,
\]
where $\mathscr{L}^\vee$ is the dual of a line bundle $\mathscr{L}$.
\newline

Here, we dare using the painfully sloppy notation for differential forms pulled back to charts, in order for the reader to become more resilient. In particular, this is going to simplify notation in many integrals. In \cite[Remark 2.2]{HearingSerre}, it is explained how to properly read this often encountered sloppy notation.

\section{Canonical measures and differential forms}

Let us summarise briefly the content of \cite[Chapitre I.2]{Oesterle1984}. Given an $O_K$-scheme $\mathfrak{X}$ which is separated, smooth, of finite type, and whose fibres are all of equal dimension $n$, observe that the space $\mathfrak{X}(O_K)$ of $O_K$-rational points of $\mathfrak{X}$ is a compact $K$-analytic $n$-manifold, open inside the algebraic variety $\mathfrak{X}(K)$. 
\newline

There exists a unique Radon measure $\nu_{\can}$ on $\mathfrak{X}(O_K)$ such that for all $n\ge1$ the fibres of the maps
\[
\kappa_k\colon \mathfrak{X}(O_K)\to\mathfrak{X}_k=\mathfrak{X}(O_K/\mathfrak{p}^k)
\]
all have measure $q^{-kn}$. This is known as the \emph{canonical measure} on $\mathfrak{X}(O_K)$ and is nowhere vanishing, cf.\ e.g.\ \cite[Chapter 2.2]{WeilAAG}, where it is shown to come from a nowhere vanishing $K$-analytic differential $n$-form, called a \emph{gauge form}. As  the $K$-analytic manifold $\mathfrak{X}(O_K)$ is compact, there always exists a gauge form on $\mathfrak{X}(O_K)$, cf.\ \cite[Théorème (2)]{Serre1965}. Its relationship to the Haar measure $\mu_{K^n}$ on $K^n$ is as follows: locally, on a Zariski-affine chart $(\mathfrak{U},\phi)$ of $\mathfrak{X}$, it holds true that
\[
\nu_{\can}|_{\mathfrak{U}(O_K)}=\phi^*\mu_{K^n},\quad\phi^*\nu_{\can}|_{\mathfrak{U}(O_K)}=\mu_{K^n}\,,
\]
and this expresses itself on fibres of $\kappa_k$ as
\begin{align}\label{RNDtrivial}
q^{kn}=\nu_{\can}(\kappa^{-1}(\bar{x}))=\mu_{K^n}(\phi(\kappa^{-1}(\bar{x})))
\end{align}
for $x\in\mathfrak{X}(O_K)$.



\begin{remark}
Let us remark here that what Serre proved in \cite{Serre1965} is that compact $p$-adic analytic varieties seem somewhat uninteresting,  the only interesting quantity apart from their dimension being the number of $p$-adic balls it decomposes into modulo $(q-1)$. However, it is extra structure unearthed from them which makes them much more interesting, e.g.\ them being open submanifolds of algebraic varieties (like here), or through an integral structure and a compatible atlas, cf.\ e.g.\ \cite{diffMfp,HearingSerre,brad_habil,BKL2026}.
\end{remark}

Let $\omega$ be an algebraic differential $n$-form on the generic fibre $X$ of $\mathfrak{X}$. It corresponds to an analytic differential $n$-form on $\mathfrak{X}(K)$, also denoted as $\omega$. By restriction it is an analytic differential $n$-form on $\mathfrak{X}(O_K)$. Now, for any $x\in\mathfrak{X}(O_K)$, the tangent space $T_x(\mathfrak{X})$ identifies with an $O_K$-lattice in the $K$-vector space $T_x(\mathfrak{X}(O_K))$, i.e.\ an \emph{integral structure} as shown in \cite[Proposition 3.2.15]{BKL2026} (where it is called $R$-structure for $R=O_K$). The assignment
\[
\mathfrak{X}(O_K)\to\Hom\left(\bigwedge^n T_x(\mathfrak{X}(O_K)),K\right)\,,\;x\mapsto \omega_{(x)}\,,
\]
where $\omega_{(x)}$ is the linear form induced by $\omega$, yields that the image
\[
\mathfrak{l}_x
=\omega_{(x)}\left(\bigwedge^n T_x(\mathfrak{X}(O_K))\right)\subseteq K
\]
is a fractional ideal of $K$. Hence, the modulus
\begin{align}\label{modulusDiffn}
\norm{\omega(x)}:=\norm{\omega_{(x)}}=\absolute{\ell_x}
\end{align}
of any generator $\ell_x\in K$ of this fractional ideal $\mathfrak{l}_x$ depends only on $\mathfrak{l}_x$.
\newline

Let $\modulus(\omega)$ be the measure on $\mathfrak{X}(O_K)$, called the \emph{modulus of $\omega$}. It can be written locally on a given chart as 
\begin{align}\label{locallyModulus}
\int_{U}h(x)\,d\modulus(\omega(x))
=\int_Uh(x)\absolute{g_U(x)}\absolute{dx}\,,
\end{align}
where $h\in\mathcal{D}(\mathcal{X}(O_K))$ is a test function, 
$d\mu(x)=\absolute{dx}$ is the Haar measure (as it appears in integrals), and  
\[
\omega|_U=g_U\,dx
\]
a local representation of the algebraic differential $n$-form with $g_U$  a polynomial having coefficients in $K$.
\newline



Then we have:

\begin{thm}[Oesterl\'e 1984]\label{OesterleTheorem84}
It holds true that
\[
\int_A d\modulus(\omega)=\mu(O_K)^n
\int_A\norm{\omega(x)}\,d\nu_{\can}(x)
\]
for any Borel set $A\subset\mathfrak{X}(O_K)$.
\end{thm}

\begin{proof}
Cf.\ \cite[Theorem 2.4]{Oesterle1984}.
\end{proof}

\subsection{Laplacian Operators via Pluricanonical forms}

Let
\[
\omega\in\Gamma\left(\mathfrak{U},\left(\Omega_{\mathfrak{X}/O_K}^n\right)^{\otimes r}\right)
\]
for   some open subscheme $\mathfrak{U}$ of $\mathfrak{X}$,
i.e.\ $\omega$ is an $r$-pluricanonical form on $\mathfrak{X}$ with poles. It gives rise to the measure $\modulus(\omega)^{\frac{1}{r}}$ which on a local chart is of the form
\[
d\modulus(\omega)^{\frac{1}{r}}|_U
=\absolute{g_U}^{\frac{1}{r}}\absolute{dx}
\]
for some meromorphic function $g_U$ on $U$.
Assume that $\rho\colon\mathfrak{X}\to\Spec(O_K)$ is proper smooth morphism of schemes and assume that 
\[
\divisor(\omega)=\sum\limits_{i=1}^s a_i\mathfrak{D}_i
\]
as a $\mathds{Q}$-divisor is also a \emph{strictly normal crossing divisor} (SNCD) over $O_K$. The latter means that all (reduced) irreducible pieces $\mathfrak{D}_i$ with $i\in I = \mathset{1,\dots, s}$
are smooth over $O_K$, and for all $x\in\supp(\divisor(\omega))$, the completion of the inclusion map
\[
\supp (\divisor(\omega)) \to\mathfrak{X} 
\]
(where the divisor support is viewed as a reduced closed subscheme) at
$x$ is isomorphic to the inclusion
\[
\Spec\left(\widehat{O}_{S,\rho(x)}[[x_1,\dots,x_n]]/(x_1\cdots x_\ell)\right)\to\Spec\left(\widehat{O}_{S,\rho(x)}[[x_1,\dots,x_n]]\right)
\]
for some $\ell$ with $1\le\ell\le n$, where $S = \Spec(O_K)$, and $n$ is the relative dimension of the structure
morphism $g \colon\mathfrak{X}\to S$. That is a way of saying that, locally in $x \in\supp(\divisor(\omega))$, the local ring
$O_{\mathfrak{X},x}$ is regular, and there exists a regular system of parameters $x_1,\dots,x_n\in\mathfrak{m}_{\mathfrak{X},x}$ such that
$\supp(\divisor(\omega))$ is locally cut out by the equation $x_1\cdots x_\ell = 0$ in $O_{\mathfrak{X},x}$,
cf.\ \cite[Definition 41.21.1]{stacks}.
The passing to the completion in the present case preserves regularity 
\cite[Proposition 2.41]{Liu2002},
and the local presentation of these completed local rings as power series rings in $n$ variables for
Noetherian $O_K$-schemes is well-known in algebraic geometry. 
\newline

A pluricanonical form $\omega$ on $\mathfrak{X}$ with poles yields a measure on $\mathfrak{X}(O_K)$ via
\[
\int_U\modulus(\omega)^{\frac{1}{r}}
=\int_U \absolute{g_U(x)}^{\frac{1}{r}}\absolute{dx}
\]
on a chart $U$ of $\mathfrak{X}(O_K)$, where $\omega$ is given as
\[
\omega|_U(x)=g_U(x)\,(dx_1\wedge\dots\wedge dx_n)^{\otimes r}\,.
\]
Since the structure morphism 
$\rho\colon\mathfrak{X}\to S$ is smooth of relative dimension $n$, it holds true that
\[
\widehat{O}_{\mathfrak{X},x}
\cong \widehat{O}_{S,\rho(x)} [[x_1,\dots,x_n]]
\]
for $x\in\divisor(\omega)$, and for any non-empty subset $J\subset I = \mathset{1,\dots, s}$, the closed subscheme
\[
\mathfrak{D}_J =
\bigcap\limits_{j\in J}\mathfrak{D}_j
\]
of $\mathfrak{X}$ is regular of codimension $\absolute{J}$ in all its points, by the SNCD assumption, cf. \cite[Lemma
41.21.2]{stacks}.
\newline

Define also
\[
\mathfrak{D}_\emptyset=\mathfrak{X}\,,\quad
\mathfrak{D}_J^\circ=\mathfrak{D}_J\setminus\bigcup\limits_{j\in I\setminus J}\mathfrak{D}_j\,.
\]
Then we have a condition for the measure $\modulus(\omega)^{\frac{1}{r}}$ to be of finite value, in which case the volume can be explicitly given: 
\begin{Lemma}\label{ItoDenefFormula}
It holds true that
\[
\modulus(\omega)^{\frac{1}{r}}
(\mathfrak{X}(O_K))=\frac{1}{q^n}
\sum\limits_{J\subset I}\absolute{\mathfrak{D}_J^\circ(\mathds{F}_q)}\prod\limits_{j\in J}\frac{q-1}{q^{(a_j/r)+1}-1}
\]
for $a_1,\dots,a_s>-r$.
\end{Lemma}

\begin{proof}
\cite[Proposition 3.5]{Ito2004}.
\end{proof}

In the proof of \cite[Proposition 3.5]{Ito2004}, we see also the following: let 
\[
\kappa\colon\mathfrak{X}(O_K)\to\mathfrak{X}(\mathds{F}_q)
\]
be the reduction map, and let $\bar{x}\in\mathfrak{X}(\mathds{F}_q)$. Then the fibre volume
\begin{align}\label{fibreVolume}
\modulus(\omega)^{\frac{1}{r}}(\kappa^{-1}(\bar{x}))
=\frac{1}{q^n}\prod\limits_{j\in I\colon\bar{x}\in\mathfrak{D}_j}\frac{q-1}{q^{(a_j/r)+1}-1}
\end{align}
classifies the types of fibres the reduction map $\kappa$ has.
\newline

Write the divisor $\divisor(\omega)$ as
\[
\divisor(\omega)=V(\omega)-P(\omega)\,,
\]
where $V(\omega)$ is the $\mathds{Q}$-divisor of zeros, and $P(\omega)$ the $\mathds{Q}$-divisor of poles of $\omega$.
The assignment
\[
\mathfrak{X}(O_K)\setminus \supp(P(\omega))\to\Hom\left(\bigwedge\limits_{i=1}^nT_x(\mathfrak{X}(O_K)),K\right),\;
x\mapsto \omega_{(x)}\,,
\]
where $\omega_{(x)}$ is the linear form induced by $\omega$, yields that the image
\[
\mathfrak{l}_x=\omega_{(x)}\left(\bigwedge\limits_{i=1}^nT_x(\mathfrak{X}(O_K))\right)
\]
is a fractional ideal of $K$ for $x\in\mathfrak{X}(O_K)\setminus \supp(P(\omega))$.
The modulus
\[
\norm{\omega(x)}:=\norm{\omega_{(x)}}=\absolute{\ell_x}
\]
of any generator $\ell_x$ of $\mathfrak{l}_x$ is well-defined and depends only on $\mathfrak{l}_x$ for $x\in\mathfrak{X}(O_K)\setminus \supp(P(\omega))$.

\begin{thm}\label{RNderFormula}
  The Radon-Nikodym derivative
  \[
\frac{d\modulus(\omega)^{\frac{1}{r}}}{d\nu_{\can}}
\]
exists and takes the value
\[
\frac{d\modulus(\omega)^{\frac{1}{r}}}{d\nu_{\can}}(x)=\lim\limits_{d_\omega(z,x)\to0}
\norm{\omega(z)}^{\frac{1}{r}}
  \]
  for $r>>0$ and any $x\in\mathfrak{X}(O_K)$.
  \end{thm}



\begin{proof}
  The differential form $\omega$ is locally on a chart $U\subset\mathfrak{X}(O_K)$ of the form
  \[
\omega|_U=g_U\,dx\,,
\]
where $g_U$ is a rational function on $U$, and the canonical measure $\nu_{\can}$ is given by an nowhere vanishing analytic differential $n$-form $\omega_{\can}$, locally on $U$ of the form
\[
\omega_{\can}|_U=h_U\,dx
\]
with $h_U$ an analytic function on $U$. Let $A\subset\mathfrak{X}(O_K)$ be a zero set w.r.t.\ $\nu_{\can}$, assumed to be contained in a compact open chart $U\subset\mathfrak{X}(O_K)$.
Then
\begin{align*}
0&=\nu_{\can}(A)=\int_A\absolute{\omega_{\can}(x)}=\int_A\absolute{h_U(x)}\absolute{dx}\,,
\end{align*}
which implies
\[
\mu_{K^n}(A)=0\,,
\]
since $\absolute{h_U}$ is locally constant on $U$, and positive. It now follows that
\[
\modulus(\omega)^{\frac{1}{r}}(A)=\int_A\absolute{g_U(x)}^{\frac{1}{r}}\absolute{dx}
\le\left[\sup_{x\in U}\absolute{g_U(x)}^{\frac{1}{r}}\right]\mu_{K^n}(A)=0\,,
\]
where the supremum on the right hand side exists in $\mathds{R}_+$, because it has an upper bound given by 
Lemma \ref{ItoDenefFormula}. This shows that the Radon-Nikodym derivative exists.

\smallskip
Now, observe that locally on an open set $V\subset\mathfrak{X}(O_K)\setminus\supp(P(\omega))$, there is a presentation
\[
\omega(x)=\ell_xdx=\omega_{(x)}(e_1\wedge\dots\wedge e_n)
\]
for the standard basis $e_1,\dots,e_n$ in each tangent space $T_x(\mathfrak{X}(O_K))$ for $x\in V$. In other words,
\[
\ell_x=\frac{d\omega_{(x)}}{dx}(e_1\wedge\dots\wedge e_n)=g_U(x)\,,
\]
if $x\in U$ for a chart $U$ of $\mathfrak{X}(O_K)$. Then
\[
\int_Ad\modulus(\omega)^{\frac{1}{r}}=\int_A\absolute{\ell_x}^{\frac{1}{r}}
=\int_A\absolute{g_U(x)}^{\frac{1}{r}}\absolute{dx}<\infty\,,
\]
and it follows that
\[
\int_Ad\modulus(\omega)^{\frac{1}{r}}
=\int_A\frac{\absolute{g_U(x)}^{\frac{1}{r}}}{\absolute{h_U(x)}}d\nu_{\can}(x)\,,
\]
since $h_U$ nowhere vanishing on $U$ means that the Radon-Nikodym derivative
\[
\frac{d\mu_{K^n}}{d\nu_{\can}}=\absolute{h_U}^{-1}
\]
exists on $U$. Due to the finiteness of the Radon measure $\modulus(\omega)^{\frac{1}{r}}$ on all of $\mathfrak{X}(O_K)$,
taking limits to $x\in P(\omega)\cap U$ is now possible, and it follows that
\[
\frac{d\modulus(\omega)^{\frac{1}{r}}}{d\nu_{\can}}(x)
=\lim\limits_{z\to x}\frac{\absolute{g_U(z)}^{\frac{1}{r}}}{\absolute{h_U(z)}}
  =\lim\limits_{z\to x}\norm{\omega(z)}^{\frac{1}{r}}\frac{d\mu_{K^n}}{d\nu_{\can}}(z)
\]
for $z\in U\setminus P(\omega)$ approaching $x$ w.r.t.\ the geodetic distance $d_\omega$ defined in the following subsection, cf.\ Lemma \ref{geodeticPDist} below.
Using (\ref{RNDtrivial}), this proves the assertions.
\end{proof}

\subsection{Ultracontractive geodetic-distance-power Operators}

Recall the definition of the geodetic distance $d_\Lambda$ on a $p$-adic analytic manifold $X$ with an integral structure $\Lambda$ from \cite{diffMfp}. It is defined via a fixed finite atlas $\mathcal{A}$ compatible with the integral structure through paths passing through the nerve complex $N(\mathcal{A})$. A face in the finite simplicial complex $N(\mathcal{A})$ corresponding to a unique compact open subset of $X$ then allows to define path length as the $\mu_\Lambda$-volume of the set corresponding to the path in $N(\mathcal{A})$, where $\mu_\Lambda$ is the Radon measure associated with the integral structure $\Lambda$. Minimising path lengths then leads to the notion of \emph{geodetic distance} $d_\Lambda$.
The effect of replacing the measure $\mu_\Lambda$ with the Radon measure $\modulus(\omega)^{\frac{1}{r}}$
on $\mathfrak{X}(O_K)$ for $r>>0$ is clarified in the following Lemma:

\begin{Lemma}\label{geodeticPDist}
The function $d_\omega(x,y)$ defined like the geodetic distance $d_\Lambda(x,y)$ on $\mathfrak{X}(O_K)$, but with measure $\mu_\Lambda=\nu_{\can}$ replaced by the measure $\modulus(\omega)^{\frac{1}{r}}$, is a distance function.
\end{Lemma}

\begin{proof}
Assume that $x\neq y$ in $\mathfrak{X}(O_K)$.
  The assertion now follows from the fact that $d_\omega(x,y)$ is defined as the $\modulus(\omega)^{\frac{1}{r}}$-measure of a clopen set containing $x$ and $y$, and that the zero sets w.r.t.\ to $\modulus(\omega)^{\frac{1}{r}}$ are of codimension at least one, and therefore not open.
\end{proof}

\begin{Lemma}\label{betaRegular}
The space $(\mathfrak{X}(O_K),d_\omega)$ is a separable, proper, $\beta$-regular locally ultrametric space for $\beta=1$.
\end{Lemma}

\begin{proof}
  The proof of \cite[Lemma 5.1]{brad_nbvp} generalises to this setting in the following way:
    First, the geodetic distance $d_\omega$ is ultrametric on each ball $B\subset\mathfrak{X}(O_K)$. Namely,
  \[
d_\omega(x,y)=\modulus(\omega)^{\frac{1}{r}}(x\wedge y)
\]
for $x,y\in B$,
where $x\wedge y$ is the smallest ball in $B$ containing $x$ and $y$. This implies
\[
d_\omega(x,y)\le\modulus(\omega)^{\frac{1}{r}}(x\wedge z\wedge y)=\max\mathset{d_\omega(x,z),d_\omega(z,y)}
\]
for $x,y,z\in B$. Hence, $d_\omega$ is a locally ultrametric distance on $\mathfrak{X}(O_K)$.
Separability of $(\mathfrak{X}(O_K),\mathcal{A},\modulus(\omega)^{\frac{1}{r}})$ is clear, and
properness follows from the fact that balls in $(\mathfrak{X}(O_K),d_\omega)$ are compact. In order to show $\beta$-regularity,
present $\omega$ locally on  a chart $(U,\phi)\in\mathcal{A}$ as
\[
\omega|_U=g_U\,dx
\]
with some rational function $g_U$ on $\phi(U)\subset K^n$. Since $\modulus(\omega)^{\frac{1}{r}}$ is a finite Radon measure on $\mathfrak{X}(O_K)$ according to
Lemma \ref{ItoDenefFormula}, it follows that
\[
\sup\limits_{x\in U}\absolute{g_U(x)}^{\frac{1}{r}}<\infty\,.
\]
This means that for a ball  $B(x,r)\subset U$, we have
\[
\modulus(\omega)^{\frac{1}{r}}(B(x,r))\le\sup\limits_{y\in U}\absolute{g_U(y)}^{\frac{1}{r}}\mu_{K^n}(\phi(B(x,r)))\,.
\]
As to a lower bound, let
\[
\mathcal{F}(x,r)=\mathset{\bar{z}\in\mathfrak{X}(\mathds{F}_q)\mid\kappa^{-1}(\bar{z})\cap B(x,r)\neq\emptyset}\,.
\]
Then there exists a constant $c>0$ such that
\begin{align*}
  \modulus(\omega)^{\frac{1}{r}}(B(x,r))
  &=c\modulus(\omega)^{\frac{1}{r}}(\kappa^{-1}(\mathcal{F}(x,r)))
  \ge c\min\limits_{\bar{z}\in\mathcal{F}(x,r)}\modulus(\omega)^{\frac{1}{r}}(\kappa^{-1}(\bar{z}))
\end{align*}
holds true. Since
\begin{align*}
\modulus(\omega)^{\frac{1}{r}}(\kappa^{-1}(\bar{z}))= C(\bar{z},r)\mu_{K^n}(\phi(\kappa^{-1}(\bar{z}))
\end{align*}
for
\[
C(\bar{z},r)=\prod\limits_{j\in I\colon\bar{z}\in\mathfrak{D}_j}\frac{q-1}{q^{(a_j/r)+1}-1}
>0
\]
as in (\ref{fibreVolume}), we can now define
\[
c_\omega=c\min_{\bar{z}\in\mathfrak{X}(\mathds{F}_q)}C(\bar{z},r)\,\quad C_\omega=\max\limits_{(U,\phi)\in\mathcal{A}}\sup\limits_{x\in U}\absolute{g_U(\phi(x))}^{\frac{1}{r}}\,,
\]
and obtain
\[
c_\omega\mu_{K^n}(B(x,r))\le\modulus(\omega)^{\frac{1}{r}}(B(x,r))\le C_\omega\mu_{K^n}(B(x,r))
\]
with the quantities $c_\omega,C_\omega$ are positive and independent
of $x\in\mathfrak{X}(O_K)$ and $r>0$.
Hence, $(\mathfrak{X}(O_K),d_\omega)$ is $\beta$-regular for $\beta=1$. This proves the assertions.
\end{proof}

Define the following Laplace integral operator on the function space $\mathcal{D}(\mathfrak{X}(O_K))$ of locally constant functions
$\mathfrak{X}(O_K)\to\mathds{C}$:
\[
\Delta_{\omega^{1/r}}^\alpha u(x)
=\int_{\mathfrak{X}(O_K)}w(x,y)(u(x)-u(y))\,d\modulus(\omega)^{\frac{1}{r}}(y)
\]
with $u\in\mathcal{D}(\mathfrak{X}(O_K))$ and kernel function
\[
w(x,y)=d_\omega(x,y)^{-\alpha}\,,
\]
where $d_\omega(x,y)$ is the geodetic distance on $\mathfrak{X}(O_K)$, defined via the measure $\modulus(\omega)^{\frac{1}{r}}$.

\begin{thm}
  Assume that $\alpha>1$.
  The semigroup $\exp\left(-t\Delta_{\omega^{1/r}}^\alpha\right)$ acting on the Hilbert space $L^2\left(\mathfrak{X}(O_K),\modulus(\omega)^{\frac{1}{r}}\right)$ is ultracontractive
  with rate function
  \[
\gamma(t)=Ct^{1(1-\alpha)}\left(1+\frac{\diam_\omega(\mathfrak{X}(O_K))}{t^{1/(\alpha-1)}}\right)^{-\alpha}
  \]
  for some $C>0$, and for $t>0$.
\end{thm}

\begin{proof}
  The proof of \cite[Theorem 5.6]{brad_nbvp}
carries over to this case, because again the 
heat kernel estimate \cite[Theorem 5.1]{Grigoryan2023}
can be used due to Lemma \ref{betaRegular}.
\end{proof}

\begin{remark}\label{heatKernelExists}
  The proof of \cite[Theorem 5.6]{brad_nbvp} proceeds by constructing a heat kernel $p_t(x,y)$ for the operator.
  This procedure is what carries over also to the case of the Radon measure $\modulus(\omega)^{\frac{1}{r}}$, here.
  It satisfies estimates leading to the rate function $\gamma(t)$ for $t>0$.
\end{remark}

\begin{cor}\label{ESelfAdjointHK}
  Assume that $\alpha>1$.
  The operator $\Delta_{\omega^{1/r}}^\alpha$ is essentially self-adjoint with a purely discrete spectrum.
  There exists a heat kernel $p_t(x,y)$ for $\exp\left(-t\Delta_{\omega^{1/r}}^\alpha\right)$ acting on $L^2\left(\mathfrak{X}(O_K),\modulus(\omega)^{\frac{1}{r}}\right)$, and for every $t_0>0$, there exist
  $c_1,c_2>0$ such that
  \[
\absolute{p_t(x,y)-\frac{1}{\modulus(\omega)^{\frac{1}{r}}(\mathfrak{X}(O_K))}}\le c_1e^{-c_2t}
\]
for all $t>t_0$ and $\modulus(\omega)^{\frac{1}{r}}$-a.e.\ $x,y\in\mathfrak{X}(O_K)$.
\end{cor}

\begin{proof}
  Since, according to Lemma \ref{ItoDenefFormula},
  \[
\modulus(\omega)^{\frac{1}{r}}(\mathfrak{X}(O_K))<\infty,
\]
the spectrum of $\Delta_\omega^\alpha$ is purely discrete, cf.\
\cite[Theorem 2.1.4]{Davies1989}. This implies essential self-adjointness.
According to Remark \ref{heatKernelExists}, there is a heat kernel, and now
the asserted inequality follows from \cite[eq.\ (18)]{HS2022}.
\end{proof}

In any case, eigenvalues for eigenfunctions supported outside the support of $\divisor(\omega)$ can be explicitly given:

\begin{Proposition}\label{smallWaveletEigenvalue}
The wavelets $\psi$ with support on a small ball  $B(a)\subset\mathfrak{X}(O_K)\setminus \supp(\divisor(\omega))$ are eigenfunctions of $\Delta_{\omega^{1/r}}^\alpha$ with eigenvalue
\begin{align*}
\lambda_{\psi}&=\int_{\mathfrak{X}(O_K)\setminus B(a)}d_\omega(x,y)^{1-\alpha}
\,d\modulus(\omega)(y)
\\
&+\left[\modulus(\omega)^{\frac{1}{r}}(B(a))\right]^{-\alpha}
\left(1-q^{-n}(1+(-1)^n)\right)
\end{align*}
for $\alpha\in\mathds{R}$.
\end{Proposition}

\begin{proof}
Since locally, the measure $\modulus(\omega)^{\frac{1}{r}}$ is of the form
\[
d\modulus(\omega)^{\frac{1}{r}}|_U=\absolute{g_U}^{\frac{1}{r}}\, d\nu_{\can}\,,
\]
and  the support of $\psi$ does not contain a zero or pole of the rational function $g_U$,
the proof in \cite[Proposition 6.1]{brad_nbvp} carries over to this case,
as long as  $\absolute{g_U}$ is constant on $\supp(\psi)$. The reason is that in this case,
the vanishing integral
\[
\int_{\supp(\psi)}\psi(x)\absolute{g_U(x)}^{\frac{1}{r}}\absolute{dx}=0
\]
can be guaranteed.
\end{proof}


\subsection{Separable maps and integral structures}

Let $f\colon \mathfrak{Y}\to \mathfrak{X}$ be a finite, separable, dominant morphism between $O_K$-schemes which are assumed separated, smooth, of finite type. The map $f$  is generically \'etale with the generic dimension being $n$. 
Assume further that $\omega_{\mathfrak{X}}$ and $\omega_{\mathfrak{Y}}$ are algebraic differential $n$-forms on the generic fibres $X$ of $\mathfrak{X}$, and $Y$ of $\mathfrak{Y}$, respectively.
Here, they are assumed without poles.
Again, denote the associated differential $n$-forms on $\mathfrak{X}(O_K)$ and $\mathfrak{Y}(O_K)$ as $\omega_X$ and $\omega_Y$, respectively.
\newline

In this situation, the composed scheme morphisms of (\ref{composedMorphism}) is thus:
\begin{align}
\xymatrix{
\mathfrak{Y}\ar[r]^f&\mathfrak{X}\ar[r]&\Spec(O_K)\,,
}
\end{align}
where the second map is the structure morphism. And the exact sequence of quasi-coherent sheaves (\ref{exactSequence_QCSh})
plays out as equivalences of line bundles or divisors:
\begin{align}\label{canonicalIsomorphism}
\omega_{\mathfrak{Y}/O_K}&\cong f^*\omega_{\mathfrak{X}}\otimes_{\mathcal{O}_{\mathfrak{Y}}}\mathcal{O}_{\mathfrak{Y}}(R)\,,
\end{align}
in which the canonical line bundles appear, and where the relative one is  $\mathcal{O}_{\mathfrak{Y}}(R)$. The divisor $R$ contains the ramification information of the morphism $f$. Hence, the name \emph{ramification divisor} for $R$. The key property of  is that $R$ is effective:

\begin{Lemma}\label{effectiveRamification}
The ramification divisor $R$ is effective. 
\end{Lemma}

\begin{proof}
This is well-known. A sketch of the proof is given here: the definition of $R$ as
\[
R=\sum\limits_{P\in\mathfrak{X}\;\text{of codim $1$}}\text{length}(\Omega_{\mathfrak{Y}/\mathfrak{X}})_P[P]
\]
obviously reveals the divisor $R$ as being effective. However, it is by proving the isomorphism (\ref{canonicalIsomorphism}) which makes its effectivity effectively meaningful. For this, the exact sequence (\ref{exactSequence_QCSh}) yields, via taking $\wedge^n$, the isomorphism
\[
\omega_{\mathfrak{Y}/O_K}\cong f^*\omega_{\mathfrak{X}/O_K}\otimes\det(\Omega_{\mathfrak{Y}/\mathfrak{X}})\,,
\]
because that exact sequence is one of locally free sheaves of finite rank on a scheme, cf.\ \cite[Corollary 4.2]{Liu2002}.
By smoothness of the morphism $f$, the latter equals
\[
\det(\Omega_{\mathfrak{X}/\mathfrak{Y}})\cong\omega_{\mathfrak{Y}/\mathfrak{X}}\,,
\]
and the isomorphism now follows from Duality Theory, cf.\ \cite[Chapter 6.4]{Liu2002}. 
\end{proof}

\begin{Proposition}\label{RN_genEtale}
The following Radon-Nikodym derivative
\[
\frac{d\modulus(\omega_{\mathfrak{Y}})}{d\modulus(f^*\omega_{\mathfrak{X}})}
\]
exists, and is of the form
\[
\frac{d\modulus(\omega_{\mathfrak{Y}})}{d\nu_{Y,\can}}
\frac{d\nu_{Y,\can}}{d\modulus(f^*\omega_{\mathfrak{X}})}
\]
on $\mathfrak{Y}(O_K)$, where the canonical measure on $\mathfrak{Y}(O_K)$ is denoted as $\nu_{Y,\can}$.
\end{Proposition}

\begin{proof}
This is the chain rule for Radon-Nikodym derivatives, and is valid if
\[
\modulus(f^*\omega_{\mathfrak{X}})<<\nu_{Y,\can}<<\modulus(\omega_{\mathfrak{Y}})
\]
holds true, where $\nu<<\nu'$ denotes  strong continuity of the measure $\nu$ with respect to the measure $\nu'$.

\smallskip
The first strong continuity relation is seen as follows:
 let $A\subset\mathfrak{Y}(O_K)$ be a Borel set with
\[
\nu_{Y,\can}(A)=0\,.
\]
Hence, as in the proof of Theorem \ref{RNderFormula}, it follows that $\mu_{K^n}(A)=0$, and thus
\begin{align*}
\int_A\absolute{f^*\omega_{\mathfrak{X}}}
&=\int_A\absolute{g_U(f(y))}\absolute{\det(T_yf)}\absolute{dy}
\stackrel{(*)}{\le}\max\limits_{y\in \mathfrak{Y}(O_K)}\absolute{g_U(f(y))}\mu_{K^n}(A)=0\,,
\end{align*}
where the inequality $(*)$ follows in this way: first, since $f$ is a morphism of $O_K$-schemes, it holds true that 
\[
\absolute{\det(T_yf)}=1
\]
for $\nu_{Y,\can}$-almost every $y\in\mathfrak{Y}(O_K)$, i.e.\ more precisely, in the locus of $f$ being \'etale. Next, the map
\[
y\mapsto\absolute{g_U(f(y))}
\]
is continuous on the closed-open chart $V$ in the compact $\mathfrak{Y}(O_K)$ and containing $A$, where $U$ is a chart in $\mathfrak{X}(O_K)$ containing $f(V)$. This proves $\modulus(f^*\omega_{\mathfrak{X}})<<\nu_{Y,\can}$.

\smallskip
The second one is seen as follows: let $A\subseteq\mathfrak{Y}(O_K)$ be a Borel set, contained in some closed-open chart $U$, and such that
\[
\int_A d\modulus(\omega_{\mathfrak{Y}})=0\,.
\]
Then
\[
0=\int_A\absolute{\omega_{\mathfrak{Y}}(y)}
=\int_A\absolute{g_U(y)}
\absolute{dy}\,,
\]
where locally
\[
\omega_{\mathfrak{Y}}|_U=g_U\,dy
\]
for the chart $U$ of $\mathfrak{Y}(O_K)$ containing $A$. The set $U$ being a compact subset of $\mathfrak{Y}(O_K)$, observe that the
map
\[
U\to\mathds{R},\;y\mapsto \absolute{g_U(y)}
\]
is continuous, and thus takes a minimum $m$ (and maximum) on $U$. Hence,
\[
0=\int_A d\modulus(\omega_{\mathfrak{Y}})\le m\,\mu(A)\le0
\]
implies that $\mu(A)=0$ for the Haar measure $\mu$ on $K^n$. But this clearly implies that $\nu_{Y,\can}(A)=0$, again by compactness of $U$. Hence, $\nu_{Y,\can}<<\modulus(\omega_{\mathfrak{Y}})$.

\smallskip 
By the Radon-Nikodym Theorem, it now follows that both derivatives
\[
\frac{d\modulus(\omega_{\mathfrak{Y}})}{d\nu_{Y,\can}}\,,\quad\frac{d\nu_{Y,\can}}{d\modulus(f^*\omega_{\mathfrak{X}})}
\]
exist, and the assertion follows by the chain rule for Radon-Nikodym derivatives.
\end{proof}

\begin{thm}\label{RN_divisor}
The following statements holds true for the morphism $f\colon\mathfrak{Y}\to\mathfrak{X}$: 
\begin{enumerate}
\item The Radon-Nikodym derivative takes the value
\[
\frac{d\modulus(\omega_{\mathfrak{Y}})}{d\modulus(f^*\omega_{\mathfrak{X}})}(y)
=\frac{\norm{\omega_{\mathfrak{Y}}(y)}}{\norm{f^*\omega_{\mathfrak{X}}(y)}}
\]
for $\nu_{Y,\can}$-almost all $y\in\mathfrak{Y}(O_K)$.
\item There is an equality
\[
\divisor\left(\frac{d\modulus(\omega_{\mathfrak{Y}})}{d\modulus(f^*\omega_{\mathfrak{X}})}\right)=R
\]
of divisors on $\mathfrak{Y}(O_K)$.

\item The map 
\[
\mathfrak{Y}(O_K)\to \mathds{R},\;
y\mapsto\frac{d\modulus(\omega_{\mathfrak{Y}})}{d\modulus(f^*\omega_{\mathfrak{X}})}(y)
\]
is regular, i.e.\ does not contain any poles.
\end{enumerate}
\end{thm}

\begin{proof}
1.
In Oesterl\'e's Theorem \ref{OesterleTheorem84}, it was shown that
\[
\frac{d\modulus(\omega_{\mathfrak{Y}})}{d\nu_{Y,\can}}=\mu(O_K)^n\norm{\omega_{\mathfrak{Y}}}\,,
\]
and in Proposition \ref{RN_genEtale}, it is shown that
\begin{align}\label{strongCont}
\modulus(f^*\omega_{\mathfrak{X}})<<\nu_{Y,\can}\,.
\end{align}
So,  the analogous Oesterl\'e formula
\[
\frac{d\modulus(f^*\omega_{\mathfrak{X}})}{d\nu_{Y,\can}}=\mu(O_K)^n\norm{f^*\omega_{\mathfrak{X}}}
\]
also holds true, and thus yields, together with the chain rule that
\[
\frac{d\modulus(\omega_{\mathfrak{Y}})}{d\modulus(f^*\omega_{\mathfrak{X}})}
=\frac{d\modulus(\omega_{\mathfrak{Y}})/d\nu_{Y,\can}}{d\modulus(f^*\omega_{\mathfrak{X}})/d\nu_{Y,\can}}
=\frac{\norm{\omega_{\mathfrak{Y}}}}{\norm{f^*\omega_{\mathfrak{X}}}}\,,
\]
as asserted.

\smallskip\noindent
2. The isomorphism (\ref{canonicalIsomorphism}) 
takes the differential form $\omega_{\mathfrak{Y}}(y)$ on $\mathfrak{Y}(O_K)$ to one which locally on a chart $U$ of $\mathfrak{Y}(O_K)$ has the form
\[
g_U(y)\omega_{\mathfrak{X}}(f(y))\,.
\]
By taking the corresponding measures on $\mathfrak{Y}(O_K)$, one now sees that 
\[
d\modulus(\omega_{\mathfrak{Y}})=\absolute{g_U(y)}\,d\modulus(\omega_{\mathfrak{X}})(f(y))
\]
on $U$. Hence, the Radon-Nikodym derivative 
equals 
\[
\frac{d\modulus(\omega_{\mathfrak{Y}})}{d\modulus(f^*\omega_{\mathfrak{X}})}=\absolute{g_U(y)}
\]
on that chart. By taking a cover of $\mathfrak{Y}(O_K)$ with charts, one obtains the equality of divisors.

\smallskip\noindent
3.
This follows immediately from Lemma \ref{effectiveRamification} and by the equality of divisors in 2. 
\end{proof}

\begin{remark}
  Notice that the Radon-Nikodym derivative associated here with the ramification divisor of a branched covering is an actual function,
  albeit real-valued. On the other hand, the ramification divisor itself is not the divisor of a $K$-valued function, as in general, it is a positive divisor. 
\end{remark}
\section{Point Detection via Neumann Boundary Value Problems}

Each of the following detection tasks on $\mathfrak{X}(O_K)$ are to be addressed via a NBVP methodology:
\begin{enumerate}
\item Detect the closed subschemes $\mathfrak{D}_j$ as well as their orders $a_j\in\mathds{Q}$ which constitute $\divisor(\omega)$.

\item Detect on which class of fibre of the reduction map $\kappa$, given by (\ref{fibreVolume}), a given ball is located.
\end{enumerate}

Given an open $\Omega\subseteq\mathfrak{X}(O_K)$,
write down the Dirchlet form:
\begin{align*}
\mathcal{E}_{\omega^{1/r},\closure_w\Omega}(u,v)
&=\int_{\closure_w\Omega}\int_{\closure_w\Omega}w(x,y)\left(\overline{u(x)}-\overline{v(x)}\right)(v(x)-v(y))
\\
&\qquad\cdot d\modulus(\omega)^{\frac{1}{r}}(y)\,d\modulus(\omega)^{\frac{1}{r}}(x)
\end{align*}
with the  Laplacian integral operator
\[
\Delta_{\omega^{1/r},\Omega}^\alpha u(x)=\int_{\mathfrak{X}(O_K)}w_\Omega(x,y))(u(x)-u(y))\,d\modulus(\omega)^{\frac{1}{r}}(y)
\]
with the kernel function
\[
w_\Omega(x,y)=\begin{cases}
w(x,y),&x,y\in\closure_w\Omega
\\
0,&\text{otherwise},
\end{cases}
\]
with,   in the present case,
\[
w(x,y)=d_\Lambda(x,y)^{-\alpha}
\]
for $\alpha>0$, and $\omega\in\Gamma\left(\mathfrak{U},(\Omega^n_{\mathfrak{X}/O_K})^{\otimes r}\right)$ for an open subscheme $\mathfrak{U}\subseteq\mathfrak{X}$ and $r>>0$.

\begin{remark}
Notice that, due to the non-locality of the kernel function $w(x,y)$, it holds true that
\[
\closure_w\Omega=\mathfrak{X}(O_K)\,,
\]
but we nevertheless remain with the viewpoint of a graphon-boundary:
\[
\delta_w\Omega=\mathfrak{X}(O_K)\setminus\Omega\,,
\]
in the present case, it is simply the complement set of $\Omega$.
This plays out in future considerations of non-local operators.
\end{remark}

In \cite{brad_nbvp}, a global  section of the analytic structure sheaf
of a compact $p$-adic analytic manifold $X$
was used in order to include zeros in the Radon measure, like this:
\[
\nu_g(x)=\absolute{g(x)}\absolute{\omega(x)}
\]
with $x\in X$, the function $g\colon X\to K$ regular analytic on $X$, and $\omega$ a nowhere vanishing differential $n$-form, where $n\in\mathds{N}$ is the dimension of $X$.
However, in the algebraic setting, there are quite often not enough global sections of the structure sheaf of a projective algebraic variety. E.g.\ the projective line $\mathds{P}^1_K$ allows only constant functions without poles. That is why in the algebraic setting, poles are not excluded from the considerations.
\newline 

Define the Neumann Boundary Value Problem (NBVP) for $(\phi,\Omega;\omega)$ as:
\begin{align}\label{NBVP_omega}
\Delta^\alpha_{\omega^{1/r}}u|_\Omega&=0,\quad N_{\omega^{1/r},\delta_w\Omega}u|_{\delta_w\Omega}=\phi,
\end{align}
where $u\in\dom(\mathcal{E}_{\omega^{1/r},\closure_w\Omega})$,
and
\[
\phi\in L^2\left(\delta_w\Omega,\modulus(\omega)^{\frac{1}{r}}\right)
\]
yields the boundary condition.

\subsection{Green kernels}

The Green function for $\Delta_{\omega^{1/r}}^\alpha$ is defined as
\[
G_{\omega^{1/r}}(x,y)=\int_0^\infty
\left(p_t^{\Omega}(x,y)-\frac{1}{\modulus(\omega)^{\frac{1}{r}}(\mathfrak{X}(O_K))}\right)\,dt\,,
\]
where $p_t(x,y)$ is the heat kernel associated with $\Delta_{\omega^{1/r}}^\alpha$.
For $\alpha>1$, we will show that it exists at least in the distributional sense.
\newline

First notice that
\[
\Delta_{\omega^{1/r}}^\alpha\colon \mathcal{D}(\mathfrak{X}(O_K))\to L^2\left(\mathfrak{X}(O_K),\modulus(\omega)^{\frac{1}{r}}\right)
\]
is a symmetric operator for $r>>0$. The Friedrichs extension of the Dirichlet form
\[
\mathcal{E}_{\omega^{1/r}}:=\mathcal{E}_{\omega^{1/r},\closure_w\Omega}
\]
exists, because it is defined via an operator.
This then provides for a closed operator $\bar{\Delta}_{\omega^{1/r}}^\alpha$ extending $\Delta_{\omega^{1/r}}^\alpha$ which is compact, self-adjoint on the closed subspace
$\dom(\mathcal{E}_{\omega^{1/r}})$ of $L^2(\mathfrak{X}(O_K),\modulus(\omega)^{\frac{1}{r}})$
for $\alpha>1$. Details of this approach, which uses Rellich's Theorem, can be found in \cite[Chapter VI]{Kato1980}.

\begin{definition}
The operator $\bar{\Delta}^\alpha_{\omega^{1/r}}$ is called a \emph{Dirichlet Laplacian} on the $p$-adic analytic manifold $\mathfrak{X}(O_K)$.
\end{definition}

\begin{Lemma}[$w$-Co-area formula]\label{wCoArea}
  Given a function $f\colon\mathfrak{X}(O_K)\to\mathds{R}$, and the set
  \[
\Omega_t=\mathset{x\in\mathfrak{X}(O_K)\mid f(x)>t}
\]
for each $t\in\mathds{R}$, the following equality
\begin{align*}
  \iint\limits_{\mathfrak{X}(O_K)^2}&w(x,y)\absolute{f(x)-f(y)}\,d\modulus(\omega)^{\frac{1}{r}}(x)\,d\modulus(\omega)^{\frac{1}{r}}(y)
  \\
&=\int_{-\infty}^\infty\iint\limits_{\partial_w\Omega_t}w(x,y)\,d\modulus(\omega)^{\frac{1}{r}}(x)\,d\modulus(\omega)^{\frac{1}{r}}(y)\,dt
\end{align*}
holds true.
\end{Lemma}

\begin{proof}
  We generalise the proof of \cite[Lemma 3.4]{Grigoryan2018} to the present setting.
  Define for $(x,y)\in\mathfrak{X}(O_K)^2$, such that $w(x,y)>0$, the interval
  \[
I_{(x,y)}=[f(x),f(y))\,,
  \]
  where it is assumed that $f(x)\le f(y)$. This is not an issue, because $w(x,y)$ is symmetric.
  Next, observe that
  \[
(x,y)\in\partial_w\Omega_t\;\Leftrightarrow\; t\in I_{(x,y)}\,.
  \]
  Namely,
\begin{align*}
  (x,y)\in\partial_w\Omega_t\;\Leftrightarrow\; x\in\delta_w\Omega\,,\;y\in\Omega_t\;\Leftrightarrow\;
  f(x)\le t,\;f(y)>t\;\Leftrightarrow\; t\in I_{(x,y)}\,.
\end{align*}
Now,
\begin{align*}
  \iint\limits_{\partial_w\Omega_t}w(x,y)&\,d\modulus(\omega)^{\frac{1}{r}}(x)\,d\modulus(\omega)^{\frac{1}{r}}(y)
  \\
  &=\iint\limits_{\mathfrak{X}(O_K)^2}w(x,y)\absolute{I_{(x,y)}}\,d\modulus(\omega)^{\frac{1}{r}}(x)\,d\modulus(\omega)^{\frac{1}{r}}(y)
  \\
  &=\iint\limits_{\mathfrak{X}(O_K)^2}\int_{-\infty}^\infty w(x,y)1_{I_{(x,y)}}(t)\,dt\,d\modulus(\omega)^{\frac{1}{r}}(x)\,d\modulus(\omega)^{\frac{1}{r}}(y)
  \\
  &=\int_{-\infty}^\infty\iint\limits_{\mathfrak{X}(O_K)^2}w(x,y)1_{I_{(x,y)}}(t)\,d\modulus(\omega)^{\frac{1}{r}}(x)\,d\modulus(\omega)^{\frac{1}{r}}(y)\,dt
  \\
  &=\int_{-\infty}^\infty\iint\limits_{\partial_w\Omega_t}w(x,y)\,d\modulus(\omega)^{\frac{1}{r}}(x)\,d\modulus(\omega)^{\frac{1}{r}}(y)\,dt\,,
\end{align*}
as asserted.
\end{proof}

Define the quantity
\[
h_w=\inf\limits_{0<\modulus(\omega)^{\frac{1}{r}}(A)\le\frac12\modulus(\omega)^{\frac{1}{r}}(\mathfrak{X}(O_K))}\mathset{\frac{\iint\limits_{\partial_w A}w(x,y)\,d\modulus(\omega)^{\frac{1}{r}}(x)\,d\modulus(\omega)^{\frac{1}{r}}(y)}{\modulus(\omega)^{\frac{1}{r}}(A)}}\,,
\]
called the \emph{Cheeger constant} of the \emph{graphon} $\left(\mathfrak{X}(O_K),w,\modulus(\omega)^{\frac{1}{r}}\right)$.

\begin{Lemma}\label{w-inequality}
  Let $f\colon\mathfrak{X}(O_K)\to\mathds{R}_+$ such that
  \[
\modulus(\omega)^{\frac{1}{r}}(\mathset{x\in\mathfrak{X}(O_K)\mid f(x)>0})\le\frac12\modulus(\omega)^{\frac{1}{r}}(\mathfrak{X}(O_K))\,.
\]
Then
\[
\iint\limits_{\mathfrak{X}(O_K)^2}w(x,y)\absolute{f(x)-f(y)}\,d\modulus(\omega)^{\frac{1}{r}}(x)\,d\modulus(\omega)^{\frac{1}{r}}(y)
\ge h_w\int\limits_{\frak{X}(O_K)}f(x)\,d\modulus(\omega)^{\frac{1}{r}}(x)
\]
holds true.
\end{Lemma}

\begin{proof}
  This is an adaptation of the proof of \cite[Lemma 3.5]{Grigoryan2018}.
  From the $w$-Co-area formula (Lemma \ref{wCoArea}), see that
  \begin{align*}
    \iint\limits_{\mathfrak{X}(O_K)^2}&w(x,y)\absolute{f(x)-f(y)}\,d\modulus(\omega)^{\frac{1}{r}}(x)\,d\modulus(\omega)^{\frac{1}{r}}(y)
\\
&    \ge
    \int_0^\infty\iint\limits_{\partial_w\Omega_t}w(x,y)\,d\modulus(\omega)^{\frac{1}{r}}(x)\,d\modulus(\omega)^{\frac{1}{r}}(y)\,dt\,.
  \end{align*}
  By assumption,
  \[
\modulus(\omega)^{\frac{1}{r}}(\Omega_t)\le\frac12\modulus(\omega)^{\frac{1}{r}}(\mathfrak{X}(O_K))
\]
for all $t\ge0$. Hence,
\[
\iint\limits_{\partial_w\Omega_t}w(x,y)\,d\modulus(\omega)^{\frac{1}{r}}(x)\,d\modulus(\omega)^{\frac{1}{r}}(y)\ge h_w\modulus(\omega)^{\frac{1}{r}}(\Omega_t)\,,
\]
and thus
\begin{align*}
  \iint\limits_{\mathfrak{X}(O_K)^2}w(x,y)\absolute{f(x)-f(y)}\,d\modulus(\omega)^{\frac{1}{r}}(x)\,d\modulus(\omega)^{\frac{1}{r}}(y)
  \ge
  h_w\int_0^\infty\modulus(\omega)^{\frac{1}{r}}(\Omega_t)\,dt\,.
\end{align*}
Now, $x\in\Omega_t$ for $t\ge0$ is equivalent with $t\in[0,f(x))$. Hence,
\begin{align*}
  \int_0^\infty&\modulus(\omega)^{\frac{1}{r}}(\Omega_t)\,dt
  =\int_0^\infty\int_{\Omega_t}\,d\modulus(\omega)^{\frac{1}{r}}(x)\,dt
  \\
  &=\int_0^\infty\int_{\mathfrak{X}(O_K)}1_{[0,f(x))}(t)\,d\modulus(\omega)^{\frac{1}{r}}(x)\,dt
    \\
    &=\int_{\mathfrak{X}(O_K)}\int_0^\infty 1_{[0,f(x))}(t)\,dt\,d\modulus(\omega)^{\frac{1}{r}}(x)
      \\
      &=\int_{\mathfrak{X}(O_K)}f(x)\,d\modulus(\omega)^{\frac{1}{r}}(x)\,,
\end{align*}
from which the assertion follows.  
\end{proof}

The Rayleigh quotient for an integral Laplacian operator $L$ with kernel function $w(x,y)$ on $L^2(E,\mu)$ for some topological space $E$ is
defined as
\[
R_w[f]=\frac{\langle Lf,f\rangle}{\langle f,f\rangle}\,,
\]
where $f\in L^2(E,\mu)$, and $\langle\cdot,\cdot\rangle$ is the inner product on $L^2(E,\mu)$.
In the case of the operator $\bar{\Delta}_{\omega^{1/r}}^\alpha$, the Rayleigh quotient is denoted as
\[
R_\alpha[f]:=R_w[f]
\]
for $f\colon\mathfrak{X}(O_K)\to\mathds{C}$ an $L^2$-function w.r.t.\ $\modulus(\omega)^{\frac{1}{r}}$.
The Rayleigh quotient belonging to an operator using a kernel function $w(x,y)$ restricted to a measurable subset $A\times A\subset E\times E$
will be denoted as
\[
R_w^A[f],\quad R_\alpha^A[f],\quad R_\alpha^\phi[f]\,,
\]
where in the second case, it is the operator $\Delta_{\omega^{1/r}}^\alpha$ being restricted, and in the third case,
\[
A=\supp(\phi)
\]
for some $L^2$-function $\phi$.
For a function $f\colon\mathfrak{X}(O_K)\to\mathds{R}$, define
\[
f_+(x)=\begin{cases}
f(x),&f(x)\ge0
\\
0,&\text{otherwise}
\end{cases}
\]
as the positive part of $f$.
The Cheeger constant of the graphon structure on $\mathfrak{X}(O_K)$ determined by $w(x,y)=d_\omega(x,y)^{-\alpha}$
will be written as
\[
h_\alpha=h_w\,,
\]
and for the subgraphon induced on $\supp(\phi)$, its notation is
\[
h_\alpha^\phi\,,
\]
where $\alpha$ is to be taken from the range $[0,\infty)$.

\begin{Lemma}\label{RayleighZero}
Let $\phi$ be a real-valued eigenfunction of the operator $\bar{\Delta}_{\omega^{1/r}}^\alpha$ with $\alpha>0$.
The following inequality
  \[
R_0^\phi[\phi_+]\ge\frac{(h_0^\phi)^2}{2\modulus(\omega)^{\frac{1}{r}}(\supp(\phi))}
\]
holds true.
\end{Lemma}

\begin{proof}
  We need to show that
\begin{align*}
  \iint_{\supp(\phi)^2}&\absolute{\phi_+(x)-\phi_-(y)}^2\,d\modulus(\omega)^{\frac{1}{r}}(x)\,d\modulus(\omega)^{\frac{1}{r}}(y)
  \\
  &\ge\frac{(h_0^\phi)^2}{2}\int_{\supp(\phi)}\phi_+(x)^2\,d\modulus(\omega)^{\frac{1}{r}}(x)\,.
\end{align*}
Since we can assume that
\[
\modulus(\omega)^{\frac{1}{r}}(\mathset{\phi_+>0})\le
\frac12\modulus(\omega)^{\frac{1}{r}}(\supp(\phi))\,,
\]
as otherwise $\phi\mapsto-\phi$ leads to this inequality.
Hence, Lemma \ref{w-inequality} implies
\begin{align}\label{suppphiIneq}
  \iint\limits_{\supp(\phi)^2}&\absolute{\phi_+(x)^2-\phi_+(y)^2}\,d\modulus(\omega)^{\frac{1}{r}}(x)\,d\modulus(\omega)^{\frac{1}{r}}(y)
  \\
  &\ge h_0^\phi \int\limits_{\supp(\phi)}\phi_+(x)^2\,d\modulus(\omega)^{\frac{1}{r}}(x)\,.
\end{align}
Now,
\begin{align*}
  \iint\limits_{\supp(\phi)^2}&\absolute{\phi_+(x)^2-\phi_+(y)^2}\,d\modulus(\omega)^{\frac{1}{r}}(x)\,d\modulus(\omega)^{\frac{1}{r}}(y)
  \\
  &=\frac12\iint\limits_{\supp(\phi)^2}
  \absolute{\phi_+(x)-\phi_+(y)}\absolute{\phi_+(x)+\phi_+(y)}\,d\modulus(\omega)^{\frac{1}{r}}(x)\,d\modulus(\omega)^{\frac{1}{r}}(y)
  \\
  &\stackrel{(*)}{\le}
  \left(\frac12\iint\limits_{\supp(\phi)^2}(\phi_+(x)-\phi_+(y))^2\,d\modulus(\omega)^{\frac{1}{r}}(x)\,d\modulus(\omega)^{\frac{1}{r}}(y)\right.
  \\
  &\qquad\cdot\left.\frac12\iint\limits_{\supp(\phi)^2}(\phi_+(x)+\phi_+(y))^2\,d\modulus(\omega)^{\frac{1}{r}}(x)\,d\modulus(\omega)^{\frac{1}{r}}(y)\right)^{\frac12}
  \\
  &\le
  \left(\;\;\iint\limits_{\supp(\phi)^2}\absolute{\phi_+(x)-\phi_+(y)}^2\,d\modulus(\omega)^{\frac{1}{r}}(x)\,d\modulus(\omega)^{\frac{1}{r}}(y)\right.
  \\
  &\qquad\cdot\left.\iint\limits_{\supp(\phi)^2}(\phi_+(x)^2+\phi_+(y)^2)\,d\modulus(\omega)^{\frac{1}{r}}(x)\,d\modulus(\omega)^{\frac{1}{r}}(y)\right)^{\frac12}
  \\
  &=\left(2\modulus(\omega)^{\frac{1}{r}}(\supp(\phi))\iint_{\supp(\phi)^2}\absolute{\phi_+(x)-\phi_+(y)}^2\,d\modulus(\omega)^{\frac{1}{r}}(x)\,d\modulus(\omega)^{\frac{1}{r}}(y)\right.
  \\
  &\qquad\cdot\left.\int_{\supp(\phi)}\phi_+(x)\,d\modulus(\omega)^{\frac{1}{r}}(x)\right)^{\frac12}\,,
\end{align*}
where $(*)$ is the Cauchy-Schwarz inequality. Inequality (\ref{suppphiIneq}) now implies
\begin{align*}
  h_0^\phi&\int_{\supp(\phi)}\phi_+(x)^2\,d\modulus(\omega)^{\frac{1}{r}}(x)
  \\
  &\le\iint\limits_{\supp(\phi)^2}\absolute{\phi_+(x)^2-\phi_+(y)^2}\,d\modulus(\omega)^{\frac{1}{r}}(x)\,d\modulus(\omega)^{\frac{1}{r}}(y)
  \\
  &\le\left(2\modulus(\omega)^{\frac{1}{r}}(\supp(\phi))\iint\limits_{\supp(\phi)^2}\absolute{\phi_+(x)^2-\phi_+(y)^2}\,d\modulus(\omega)^{\frac{1}{r}}(x)\,d\modulus(\omega)^{\frac{1}{r}}(y)\right)^{\frac12}
  \\
  &\qquad\cdot\left(\int_{\supp(\phi)}\phi_+(x)^2\,d\modulus(\omega)^{\frac{1}{r}}(x)\right)^{\frac12}\,.
\end{align*}
Squaring now allows to see the validity of the assertion.
\end{proof}

\begin{Lemma}\label{RayleighAlpha}
  It holds true that
  \[
R_\alpha[\phi_+]\ge\frac18\modulus(\omega)^{\frac{1}{r}}(\supp(\phi))^{1-\alpha}\,,
\]
where $\phi$ is as in Lemma \ref{RayleighZero}, and $\alpha>0$.
\end{Lemma}

\begin{proof}
  We have
  \[
R_\alpha[\phi_+]\ge\modulus(\omega)^{\frac{1}{r}}(\supp(\phi))^{-\alpha} R_0^\phi[\phi_+]\,.
\]
Namely,
\begin{align*}
\iint\limits_{\mathfrak{X}(O_K)^2}&d_\omega(x,y)^{-\alpha}\absolute{\phi_+(x)-\phi_+(y)}^2\,d\modulus(\omega)^{\frac{1}{r}}(x)\,d\modulus(\omega)^{\frac{1}{r}}(y)
\\
&\ge\iint\limits_{\supp(\phi)^2}d_\omega(x,y)^{-\alpha}\absolute{\phi_+(x)-\phi_+(y)}^2\,d\modulus(\omega)^{\frac{1}{r}}(x)\,d\modulus(\omega)^{\frac{1}{r}}(y)
\\
&\ge\modulus(\omega)^{\frac{1}{r}}(\supp(\phi))^{-\alpha}
\iint\limits_{\supp(\phi)^2}\absolute{\phi_+(x)-\phi_+(y)}^2\,d\modulus(\omega)^{\frac{1}{r}}(x)\,d\modulus(\omega)^{\frac{1}{r}}(y)\,.
\end{align*}
Hence, using Lemma \ref{RayleighZero}, it follows that
\begin{align*}
  R_\alpha[\phi_+]
  \ge\modulus(\omega)^{\frac{1}{r}}(\supp(\phi))^{-\alpha}\frac{(h_0^\phi)^2}{2\modulus(\omega)^{\frac{1}{r}}(\supp(\phi))}\,.
\end{align*}
Since
\[
h_0^\phi=\frac12\modulus(\omega)^{\frac{1}{r}}(\supp(\phi))\,,
\]
the assertion now follows.
\end{proof}

\begin{Proposition}\label{evLowerBd}
Let $\phi$ be a real-valued eigenfunction of the operator $\bar{\Delta}_{\omega^{1/r}}^\alpha$ with eigenvalue $\lambda_\phi$, and let $\alpha>1$.
Then there exists $C>0$ such that
\[
\lambda_\phi\ge \frac{q^{n(k-1)(\alpha-1)}}{8}\left(\frac{q-1}{q^{(a_{\min}/r)+1}-1}\right)^{s(1-\alpha)}\,,
\]
where
\[
a_{\min}=\min\mathset{a_1,\dots,a_s}\,,
\]
$\phi$ is any non-constant eigenfunction of $\bar{\Delta}_{\omega^{1/r}}^\alpha$,
and where
$\supp(\phi)$ is contained in disjoint union of as few copies as possible of a  ball
of largest possible Haar measure in $K^n$ equal to $q^{-nk}$, where $k\in\mathds{Z}$.
\end{Proposition}

\begin{proof}
First, we show that
\[
\lambda_\phi\ge R_\alpha[\phi_+]\,.
\]
Namely,
\begin{align*}
\lambda_\phi\langle\phi,\phi_+\rangle&=
\left\langle\bar{\Delta}^\alpha_{\omega^{\frac{1}{r}}}\phi,\phi_+\right\rangle
\\
  &=\frac12\iint\limits_{\mathfrak{X}(O_K)^2}d_\omega(x,y)^{-\alpha}(\phi(x)-\phi(y))(\phi_+(x)-\phi_+(y))\,d\modulus(\omega)^{\frac{1}{r}}(x)\,d\modulus(\omega)^{\frac{1}{r}}(y)
\end{align*}
Since
\[
\phi\phi_+=\phi_+^2\,,
\]
we have
\begin{align*}
  (\phi(x)-\phi(y))(\phi_+(x)-\phi_+(y))
  \ge(\phi_+(x)-\phi_+(y))^2\,.
\end{align*}
It follows that
\begin{align*}
  \lambda_\phi=\frac{\langle\bar{\Delta}_{\omega^{1/r}}\phi,\phi\rangle}{\langle\phi,\phi_+\rangle}
  \ge\frac{\langle\bar{\Delta}_{\omega^{1/r}}\phi_+,\phi_+\rangle}{\langle\phi,\phi_+\rangle}
  =\frac{\langle\bar{\Delta}_{\omega^{1/r}}\phi_+,\phi_+\rangle}{\langle\phi_+,\phi_+\rangle}
  =R_\alpha[\phi_+]
\end{align*}
Using Lemma \ref{RayleighAlpha}, we can now see that
\[
\lambda_\phi\ge\frac18\modulus(\omega)^{\frac{1}{r}}(\supp(\phi))^{1-\alpha}\,.
\]
Since each fibre volume satisfies
\[
\modulus(\omega)^{\frac{1}{r}}(\kappa^{-1}(\bar{x}))\le  q^{-n}\left(\frac{q-1}{q^{(a_{\min}/r)+1}-1}\right)^s\,,
\]
set $C>0$ as
\[
C=N^{1-\alpha}
\]
with $N$
the number of copies of Haar-largest possible balls disjointly covering $\supp(\phi)$.
This then leads to
\[
\lambda_\phi\ge\frac{C}{8}q^{nk(\alpha-1)}\left(\frac{q-1}{q^{(a_{\min}/r)+1}-1}\right)^{s(1-\alpha)}\,.
\]
The constant $C>0$
can be bounded as follows: notice that
we have
  \[
N\le q^n\,,
\]
since otherwise, the ball whose copies are used for covering $\supp(\phi)$ is not maximal.
Hence,
\[
C\ge q^{n(1-\alpha)}\,,
\]
and the asserted inequality now follows.
\end{proof}

\begin{cor}
  The spectral gap $\lambda_1$ of $\Delta_{\omega^{1/r}}^\alpha$ satisfies
  \[
\lambda_1\ge \frac{q^{n(M-1)(\alpha-1)}}{8}\left(\frac{q-1}{q^{(a_{\min}/r)+1}-1}\right)^{s(1-\alpha)}
\]
for $\alpha>1$, where $q^{-Mn}$ for some $M\in\mathds{Z}$ is the $\nu_{\can}$-volume of the $\nu_{\can}$-largest ball inside $\mathfrak{X}(O_K)$.
\end{cor}

\begin{proof}
This is an immediate consequence of Proposition \ref{evLowerBd} in light of equality (\ref{RNDtrivial}).
\end{proof}


  Corollary \ref{ESelfAdjointHK} says that the spectrum of $\Delta_{\omega^{1/r}}^\alpha$ is discrete for $\alpha>1$.
  This means that the heat kernel can be expressed as
  \[
p_t(x,y)=\sum\limits_\phi e^{-t\lambda_\phi}\phi(x)\overline{\phi(y)}
\]
for $x,y\in\mathfrak{X}(O_K)$, and where $\phi$ runs through an orthonormal basis
of
\[
L^2\left(\mathfrak{X}(O_K),\modulus(\omega)^{\frac{1}{r}}\right)
\]
of eigenfunctions of $\bar{\Delta}_{\omega^{1/r}}^\alpha$. From Proposition \ref{evLowerBd}, it follows that
\begin{align}\label{GreenBound}
\sum\limits_\phi\lambda_\phi^{-1}\le
\frac{q^{n(1-\alpha)}}{8}\left(\frac{q-1}{q^{(a_{\min}/r)+1}-1}\right)^{s(\alpha-1)}\sum\limits_{k=M}^\infty m(k) q^{-nk(\alpha-1)}\,,
\end{align}
where $m(k)$ is the number of eigenvalues $\lambda_\phi$ in the interval
\[
\left[
  \frac{q^{n(k-1)(\alpha-1)}}{8}\left(\frac{q-1}{q^{(a_{\min}/r)+1}-1}\right)^{s(1-\alpha)},
  \frac{q^{nk(\alpha-1)}}{8}\left(\frac{q-1}{q^{(a_{\min}/r)+1)}-1}\right)^{s(1-\alpha)}
  \right)\,.
\]
Hence, in order to prove that
\[
G_{\omega^{1/r}}(x,y)=\sum\limits_{\phi\colon\lambda_\phi>0}\lambda_\phi^{-1}\phi(x)\overline{\phi(y)}
\]
for $x,y\in\mathfrak{X}(O_K)$ is a well-defined function of some sort, some knowledge about the asymptotic behaviour of the
quantity $m(k)$ for $k\to\infty$ needs yet to be obtained. In any case, we believe the eigenvalue multiplicities of $\Delta_{\omega^{1/r}}^\alpha$ for $\alpha>1$ to be bounded.
Using (\ref{GreenBound}), this would imply the existence of the Green function as an $L^2$-function on $(\mathfrak{X}(O_K)^2,\modulus(\omega)^{\frac{1}{r}}\wedge\modulus(\omega)^{\frac{1}{r}})$.
\newline

Nevertheless, the Green kernel $G_{\omega^{1/r}}$ can be shown to exist for $\alpha>1$ via viewing the integral
\[
G_{\omega^{1/r}}(x,y)=\int_0^\infty\left( p_t(x,y)-\frac{1}{\modulus(\omega)^{\frac{1}{r}}(\mathfrak{X}(O_K))}\right)\,dt
\]
as a distribution,
and this can be used for solving the NBVP (\ref{NBVP_omega}).
\newline

Define the following space:
\[
\mathcal{N}=\mathset{u\in \dom(\mathcal{E}_{\omega^{1/r}})\cap L^\infty(\mathfrak{X}(O_K))\mid \Delta_{\omega^{1/r}}u\in L^\infty(\mathfrak{X}(O_K))}\,,
\]
and similarly as for the case of a measure obtained from a nowhere vanishing analytic $n$-form, one can prove that $\mathcal{N}$ is dense in
$\dom(\mathcal{E}_{\omega^{1/r}})$, cf.\ \cite[Lemma 4.5]{brad_nbvp}.

\begin{thm}\label{NBVP_alg}
  Let $\Omega\subset\mathfrak{X}(O_K)$ be open, and let $\phi\in L^\infty(\delta_w\Omega)$ such that
  \[
\int_{\delta_w\Omega}\phi(x)=0
  \]. 

  Let $\omega$ be a regular $O_K$-algebraic differential form on $\mathfrak{Y}$.
Then the NBVP (\ref{NBVP_omega}) with $\phi\in L^\infty(\delta_w\Omega,\modulus(\omega))$ has a unique solution $u\in\dom(\mathcal{E}_{\omega,\Omega})$ such that
\begin{align*}
\int_{\delta_w\Omega}&u(y)\,d\modulus(\omega)(y)=0
\end{align*}
and
\begin{align*}
u(x)&=\int_{\delta_w\Omega}\phi(y)G_{\omega,\Omega}(x,y)\,d\modulus(\omega)(y)\,.
\end{align*}
Then the following statements hold true:
\begin{enumerate}
\item There is a unique function $u\in L^\infty(\mathfrak{X}(O_K))$ that weakly solves the Neumann Boundary
  Value Problem for $(\phi,\Omega)$, and satisfies
  \[
\int_{\closure_w\Omega}u(x)\,d\modulus(\omega)^{\frac{1}{r}}(x)=0\,.
\]
\item This function satisfies
  \[
u(x)=\int_{\delta_w\Omega}G_{\omega^{1/r}}(x,y)\phi(y)\,d\modulus(\omega)^{\frac{1}{r}}(y)
\]
for all $x\in\delta_w\Omega$.
\end{enumerate}

\end{thm}

\begin{proof}
The proof of \cite[Theorem 5.8]{brad_nbvp}
carries over to this case. It uses the density of $\mathcal{N}$ in $\dom(\mathcal{E}_{\omega^{1/r}})$. 
\end{proof}

\subsection{Detecting the divisor of the differential form}

The searching task here  presumes, in particular in the higher-dimensional case, that the simplicial complex structure given by the divisor of the differential form in question is present in order to ``guide'' searchers along the poset structure of its faces, but without (explicitly) revealing one's position in it. In the case of dimension one, treated in the last two subsections, this plays out only in the task of detecting which one of the points one is or is not coming closer to when sampling different $p$-adic balls with wavelets supported on them for using them as boundary conditions of a Neumann Boundary Value Problem.
\newline

The following specialisation of (\ref{NBVP_omega}) will be chosen:
\begin{align}\label{specialisation_plurican}
\Omega=\Omega_{z,k}&=\mathfrak{X}(O_K)\setminus(B_k(z)\cup \supp(\divisor(\omega))\,,
\\\nonumber
\delta_w\Omega&=B_k(z)\cup\supp(\divisor(\omega))\,,
\\\nonumber
\phi(x)&=\psi_{B_k(z),j}(x)\,,
\end{align}
where $B_k(z)\subset\mathfrak{X}(O_K)\setminus\supp(\divisor(\omega))$ is a ball centred in $a$ of canonical (i.e.\ w.r.t.\ $\nu_{\can}$) volume $q^{nk}$

\begin{thm}
Solving the NBVP (\ref{NBVP_omega}) under specialisation (\ref{specialisation_plurican}) weakly for lots of increasing $k\in\mathds{N}$ detects 
the  volume of the fibre $\kappa^{-1}(\bar{z})$, the corresponding orders $a_{j_1},\dots,a_{j_m}$, as well as
the set of possible smallest closed subscheme $\mathfrak{D}_{\mathset{j_1,\dots,j_m}}$ containing $\bar{z}$ for $a\in\mathfrak{X}(O_K)$, and finally the divisor $\divisor(\omega)$.
\end{thm}

\begin{proof}
Since the eigenvalue $\lambda_{\psi}$ of $\psi$ is of the order of the second summand of the formula in Proposition \ref{smallWaveletEigenvalue}, it follows that sampling weak solutions $u$ of (\ref{NBVP_omega}) under specialisation (\ref{specialisation_plurican}) for increasing sufficiently large $k\in\mathds{N}$, and then taking the logarithm of $\norm{u}_\infty$  with basis $q$ allows to estimate the offset
\[
n\alpha\sum\limits_{i=1}^m
\left(\log_q(q^{a_{j_i}+1}-1)-\log_q(q-1)\right)
\]
which immediately yields the fibre volume.

\smallskip
Taking enough samples as above for $z\in\mathfrak{X}(O_K)$ in sufficiently many ``areas'' of the manifold $\mathfrak{X}(O_K)$ yields a system of ``almost'' linear equations containing the $a_i$'s on the left hand side, and the ball volume values on the right hand side. The number of different equations possible is $2^s$, for the parameters $a_1,\dots,a_s$.
Since the maximal possible rank of this approximately linear system is the number of distinct values of the $a_i$'s, taking enough samples  $z\in\mathfrak{X}(O_K)\setminus\supp(\divisor(\omega))$ with coverage of enough faces of the intersection complex associated with $\divisor(\omega)$, yields a set of approximate linear equations with this rank. This yields the values of $a_1,\dots,a_s\in\mathds{Q}$.

\smallskip
Having established the knowledge of the $a_i$'s, the question of where on the intersection complex a given ball $B_k(z)$ is taken from, can now be answered as follows: The  volume of $B_k(z)$ can now be obtained with various sets of $a_{j_1},\dots,a_{j_m}$ with sets $\mathset{j_1,\dots,j_m}$ of varying size.
Collect all of these sets into a set $K_0$. Now, sample a ball of equal size from a neighbouring face in the intersection complex. This amounts to either adding an $a_i$ to, or subtracting an $a_i$ to these sets. Collect all the new possible subsets of $I$ in a new set $K_1$, and discard 
from $K_0$ all sets which do not have a counterpart in $K_1$ be either adding or removing an element. Do the same with $K_1$. Now sample a ball of equal radius from a neighbouring face to  the previously sampled ball, and collect all the possible subsets of $I$ in $K_{n+1}$. Remove from $K_n$ any set which does not have a match in $K_{n+1}$ by adding or removing an element. Likewise with $K_{n+1}$. From there backtrack all previous sets $K_0,\dots,K_n$ by removing items not forming edit-neighbour pairs. Eventually, $K_0$ has a minimal size. We need to prove that the minimal size of $1$ is attainable, and this then yields the position of the initially sampled ball $B_k(z)$ in the intersection complex.

\smallskip
The divisor $\divisor(\omega)$ is built up by attaching to the  known values of $a_1,\dots,a_s$ the fibres of all points of the reduction $\mathfrak{D}(\mathds{F}_q)$, taking into account the faces of the intersection complex they are attached to. These are not known per se, but two fibres belonging to the same face can be reached by a trivial path on that simplicial complex. 
\end{proof}

\subsection{Detecting Weierstrass points}

Here, and in the following subsection, the algebraic differential form may have zeros, but no poles.
\newline

Let $\mathfrak{X}$ be a projective algebraic curve over $O_K$ with $g\ge2$ its genus. For each point in $\mathfrak{X}(O_K)$ take a local coordinate $z$, and let $\phi_1(z)dz,\dots,\phi_g(z)dz$ be a local presentation of a basis of algebraic differential $1$-forms on $\mathfrak{X}$. The  following determinant: 
\[
W(z)=\det
\begin{pmatrix}
\phi_1(z)&\phi_1'(z)&\dots&\phi_1^{(g)}(z)
\\
\vdots&\ddots&\dots&\vdots
\\
\phi_g(z)&\phi_g'(z)&\dots&\phi_g^{(g)}
\end{pmatrix}
\]
is the \emph{Wronskian} of $\mathset{\phi_1(z),\dots,\phi_g(z)}$ in a local neighbourhood of $x$. The order of vanishing of a point $x$ in the zero set of $W(z)$ is called the  \emph{weight} of $x$.
This is important for characterising \emph{Weierstrass points} of the curve $\mathfrak{X}$. These are points $P\in\mathfrak{X}$ for which
\[
\dim_K O_{\mathfrak{X}}(gP)\ge 2\,,
\]
i.e.\ for which there exists a meromorphic function on $\mathfrak{X}$ having a unique pole in $P$ of order at most $g$.

\begin{Lemma}
A point $x\in\mathfrak{X}(O_K)$is a  Weierstrass point
if and only if the Wronskian $W(z)$ 
in a local neighbourhood of $x\in\mathfrak{X}(O_K)$ 
vanishes at $x$.
\end{Lemma}

\begin{proof}
\cite[Proposition 11.1]{KLP2018}.
\end{proof}

The Wronskian characterisation  leads to a direct application of closeness to zero Theorem \cite[Theorem 6.3]{brad_nbvp}. Namely, let
\[
d\nu_W(z)=\absolute{W(z)}d\nu_{\can}(z)
\]
be a measure on $\mathfrak{X}(O_K)$ which incorporates the Wronskian locally in $x\in\mathfrak{X}(O_K)$.
The Dirichlet form:

\begin{align*}
\mathcal{E}_{W,\closure_w\Omega}(u,v)
&=\frac12\int_{\closure_w\Omega}\int_{\closure_w\Omega}w(x,y)\left(\overline{u(x)}-\overline{u(y)}\right)(v(x)-v(y))
\,d\nu_W(y)\,d\nu_W(x)
\end{align*}
is associated with the following Laplacian integral operator:
\[
\Delta_{W,\Omega}^\alpha
u(x)=\int_{\mathfrak{X}(O_K)}
w_\Omega(x,y)(u(x)-u(y))\,d\nu_W(y)
\]
with kernel function
\[
w\colon\mathfrak{X}(O_K)\times\mathfrak{X}(O_K)\to\mathds{R},\;
(x,y)\mapsto
d_\Lambda(x,y)^{-\alpha}
\]
for $\alpha>0$.
\newline

Let us formulate a NBVP suitable for detecting Weierstrass points on $\mathfrak{X}(O_K)$: let
$\phi\in L^2(\delta_w\Omega,\nu_{W})$. Then $u\in\dom(\mathcal{E}_{W,\closure_w\Omega})$ is a weak solution of the NBVP (\ref{NBVP_Weierstrass}), if
\begin{align}\label{NBVP_Weierstrass}
\langle \Delta_{W,\Omega}^\alpha u,v\rangle_{L^2}=0,\quad\langle N_{W,\delta_w\Omega} u,v\rangle_{L^2}=\langle\phi,v\rangle_{L^2}
\end{align}
for all $v\in\dom(\mathcal{E}_{W,\closure_w\Omega})$ supported in $\Omega$ and in $\delta_w\Omega$, respectively. Recall the condition 
\begin{align}\label{boundaryMean}
\int_{\delta_w\Omega}\phi\,d\nu_W=0
\end{align}
necessary for the existence of a weak solution, cf.\ \cite[Proposition 4.4]{brad_nbvp}. This allows for using the specialised choices:
\begin{align*}
\Omega=\Omega_{a,k}&=\mathfrak{X}(O_K)\setminus (B_k(a)\cup V(W),&B_k(a)\subset \mathfrak{X}(O_K)\setminus V(W)\,,
\\
\delta_w\Omega&=B_k(a)\cup V(W)\,,
\\
\phi(x)&=\psi_{B(a),j}(x)\,,
\end{align*}
where the latter is a wavelet on $\mathfrak{X}(O_K)$ supported in $B_k(a)$, and having parameter $j\in\mathds{F}_q$. We assume that the support $B(a)$ of $\phi$ is small. In this case, the necessary condition (\ref{boundaryMean}) is satisfied, cf.\ the proof of \cite[Propositon 6.1]{brad_nbvp}.

\begin{thm}[Closeness to Weierstrass Points]\label{close2Weierstrass}
The norm $\norm{u}_\infty$ of the weak solution $u \in \dom(\mathcal{E}_{W,\closure_w\Omega})$ of the
Neumann boundary value problem (\ref{NBVP_Weierstrass}) informs for varying $a\in \mathfrak{X}(O_K)\setminus V(W)$ about whether or
not $a$ is approaching some Weierstrass point of $\mathfrak{X}(O_K)$
 of weight $r > 0$.
\end{thm}

\begin{proof}
This is an application of \cite[Theorem 6.3]{brad_nbvp}.
\end{proof}

\subsection{Weierstrass points on hyperelliptic curves}

A hyperelliptic  curve is endowed with a covering map of degree 2:
\[
\iota\colon\mathfrak{X}\to\mathds{P}^1_K
\]
given by the hyperelliptic involution.
\newline

From the algebraic Radon-Nikodym theory in Theorem \ref{RN_divisor}, we have
\[
d\modulus(\omega_{\mathfrak{X}})=\frac{d\modulus(\omega_{\mathfrak{X}})}{d(\iota^*\nu_{\mathds{P}^1,\can})}\,d\left(\iota^*\nu_{\mathds{P}^1,\can}\right)
=\frac{\norm{\omega_{\mathfrak{X}}}}{\norm{\iota^*\nu_{\mathds{P}^1,\can}}}\,d\left(\iota^*\nu_{\mathds{P}^1,\can}\right)\,,
\]
where $\nu_{\mathds{P}^1,\can}$ is the canonical  measure on $\mathds{P}^1(O_K)$.
Notice that, according to  Corollary \ref{RN_divisor}, Statement 2., that the zeros of the Radon-Nikodym derivative
\[
F=\frac{d\modulus(\omega_{\mathfrak{X}})}{d\modulus(\iota^*\nu_{\mathds{P}^1,\can})}=\frac{\norm{\omega_{\mathfrak{X}}}}{\norm{\iota^*\nu_{\mathds{P}^1,\can}}}
\]
coincide with the ramification points of the covering map $\iota$. In particular, this function is without poles on the curve $\mathfrak{X}(O_K)$.
\newline

The Dirichlet form
\begin{align*}
\mathcal{E}_{\iota,\Omega}(u,v)
&=\frac12\int_{\closure_w\Omega}\int_{\closure_w\Omega}w(x,y)\left(\overline{u(x)}-\overline{u(y)}\right)(v(x)-v(y))
\, d\nu_{\mathds{P}^1,\can}(\iota(y))\,d\nu_{\mathds{P}^1,\can}(\iota(x))
\end{align*}
for $\Omega\subseteq \mathfrak{X}(O_K)$ open comes from the integral Laplacian
\[
\Delta_{\iota,\Omega}u(x)=\int_{\mathfrak{X}(O_K)}w_\Omega(x,y)(u(x)-u(y))\,d\nu_{\mathds{P}^1,\can}(\iota(y))
\]
and gives rise to the \emph{hyperelliptic NBVP}: 
\begin{align}\label{NBVP_iota}
\langle\Delta_{\iota,\Omega} u,v\rangle_{L^2\left(\Omega,\iota^*\nu_{\mathds{P
}^1,\can}\right)}=0,\quad \langle N_{\iota^*dz,\delta_w\Omega} u,v\rangle_{L^2\left(\delta_w\Omega,\iota^*\nu_{\mathds{P}^1,\can}\right)}=\langle\phi,v\rangle_{L^2}
\end{align}
with $\phi\in L^2\left(\closure_w\Omega,\iota^*\nu_{\mathds{P}^1,\can}\right)$, and weak solution $u\in\dom(\mathcal{E}_{\iota,\Omega})$.
Here, the specialised choices have this form:
\begin{align*}
\Omega=\Omega_{a,k}&=\mathfrak{X}(O_K)\setminus (V(F)\cup B_k(a))\,,&B_k(a)\subset\mathfrak{X}(O_K)\setminus V(F)\,,
\\
\delta_w\Omega&=B_k(a)\cup V(F)\,,
\\
\phi(x)&=\psi_{B(a),j}(x)
\end{align*}
with $j\in\mathds{F}_q$ the wavelet parameter.
\begin{thm}[Closeness to Hyperelliptic Weierstrass Points]
The norm $\norm{u}_\infty$ of the solution $u\in\dom(\mathcal{E}_{\iota,\Omega})$ of the Neumann boundary value problem (\ref{NBVP_iota}) informs for varying $a\in\mathfrak{X}(O_K)\setminus V(F)$ about whether or not $a$ is approaching a Weierstrass point of the hyperelliptic curve $\mathfrak{X}(O_K)$ of weight $r>0$.
\end{thm}

Even if one can view this as a Corollary of Theorem \ref{close2Weierstrass}, the proof here is going to use the hyperelliptic cover $\iota\colon\mathfrak{X}\to\mathds{P}^1_K$.

\begin{proof}
The Weierstrass points of a hyperelliptic curve are given by the zeros of the ramification divisor of the map $\iota$. This ramification divisor which coincides with the  zeros of the Radon-Nikodym derivative
\[
\frac{d\modulus(\omega_{\mathfrak{X}})}{d\left(\iota^*\nu_{\mathds{P}^1,\can}\right)}
\]
according to Theorem \ref{RN_divisor}, Statement 2. The assertion now follows from Closeness to Zeros, cf.\ \cite[Theorem 6.3]{brad_nbvp}.
\end{proof}
\section*{Acknowledgements}


Andrew Bradley, John Bradley,
Oleksandra Antoniouk,
Frank Herrlich, Stefan K\"uhnlein, Peer Kunstmann, \'Angel Mor\'an Ledezma and 
David Weisbart are thanked for valuable discussions. 

\section*{Conflict of Interests}

No conflicts of interests are declared.

\section*{Data Availability}

No data is available to be shared.

\bibliographystyle{plain}
\bibliography{biblio}

@article{Serre1965,
author={J.-P. Serre},
title={Classiﬁcation des variétés analytiques $p$-adiques compactes},
journal={Topology}, 
volume={3},
pages={409--412}, 
year={1965}
}

@book{WeilAAG,
author={A. Weil},
title={Adeles and Algebraic Groups}, series={Progress in Mathematics 23},
publisher={Birkhäuser}, 
address={Boston}, 
year={1982}
}

@article{DiffMfp,
title={Diffusion operators on $p$-adic analytic manifolds},
author={P.E. Bradley},
OPThowpublished={arXiv:2510.22563 [math.AP]},
year={2026},
journal={Results in Mathematics},
volume={81}, 
pages={129} 
}

@book{Davies1989,
author={E.B. Davies},
title={Heat kernels and spectral theory}, 
volume={92},
series={of Cambridge
Tracts in Mathematics},
publisher={Cambridge University Press}, 
address={Cambridge}, 
year={1989}
}

@book{Liu2002,
author={Q. Liu},
title={Algebraic Geometry and Arithmetic Curves},
series={Oxford Graduate Texts in Mathematics, vol. 6},
publisher={Oxford University Press},
address={Oxford},
year={2002},
note={Translated by R. Ern\'e}
}

@book{KLP2018,
author={M.E. Kazaryan and S.K. Lando
and V.V. Prasolov},
title={Algebraic
Curves. Towards Moduli Spaces},
volume={2},
series={Moscow Lectures},
publisher={Springer Nature Switzerland AG},
note={Translated from the Russian by Natalia Tsilevich},
year={2018},
address={Cham}
}

@article{Ito2004,
author={T. Ito},
title={Stringy {Hodge} numbers and $p$-adic {Hodge} theory},
journal={Compositio
Math.}, 
volume={140},
pages={1499--1517}, 
year={2004}
}

@misc{brad_nbvp,
author={P.E. Bradley},
howpublished={arXiv:2609.26180 [math.AP]},
year={2026},
title={Neumann Boundary Value Problems on $p$-Adic Analytic
Manifolds}
}

@article{HearingSerre,
author={P.E. Bradley and \'A.M. Ledezma},
title={Hearing the {Serre} invariant of a compact $p$-adic analytic manifold},
OPTnote={arXiv:2511.20631 [math.NT]},
year={2026},
journal={Mathematische Nachrichten},
volume={299},
pages={3126--3138}
}

@misc{brad_habil,
author={P.E. Bradley},
title={Compact $p$-Adic Analytic Manifolds
and Applications in Arithmetic Geometry},
howpublished={Habilitation Thesis},
year={in preparation}
}

@article{AK2026,
journal={Pseudo-Differ. Oper. Appl.}, year={2026},
volume={17},
pages={5},
title={Non-{Archimedean Neumann} problem: weak and strong
solutions},
author={A.V. Antoniouk and A.N. Kochubei}
}

@book{Schneider2011,
author={Schneider, P.},
title={$p$-adic {Lie} groups}, series={Grundlehren der mathematischen Wissenschaften 344},
publisher={Springer}, 
address={Berlin}, 
year={2011}
}

@book{Serre1992,
author={Serre, J.-P.},
title={Lie Algebras and {Lie} Groups. Lectures given at {Harvard University}},
series={Lecture Notes in Mathematics 1500},
publisher={Springer}, 
year={1992}
}

@book{Igusa2001,
author={J.-I. Igusa.},
title={An introduction to the theory of local {Zeta} functions}, 
volume={14},
series={AMS/IP studies in advanced mathematics},
publisher={American Mathematical
Society, International Press}, 
year={2002}
}

@article{Grigoryan2023,
title={Analysis on ultra-metric spaces via heat kernels},
author={A. Grigor’yan},
journal={$p$-Adic Numbers, Ultrametric Analysis and Applications}, 
year={2023}, 
volume={15}, 
number={3}, 
pages={204--242}
}

@book{Grigoryan2018,
author={A. Grigor'yan},
title={Introduction to Analysis on Graphs},
publisher={AMS},
series={University Lecture Series},
volume={71},
year={2018}
}

@article{HR2008,
journal={Invent. math.},
volume={174}, 
pages={555--624},
year={2008},
title={Mixed {Hodge} polynomials of character varieties.
With an appendix by Nicholas M. Katz},
author={T. Hausel and F. Rodriguez-Villegas}
}

@book{Kato1980,
author={T. Kato},
title={Perturbation Theory
for Linear Operators},
publisher={Springer},
series={Grundlehren der mathematischen Wissenschaften},
volume={132},
address={Berlin},
edition={2nd},
year={1980}
}

@article{Batyrev1999,
author={Batyrev, V.V.},
title={Non-{Archimedean} integrals and stringy {Euler} numbers of log-terminal
pairs},
journal={J. Eur. Math. Soc.},
volume={1}, 
pages={5--33},
year={1999}
}

@article{GWZ2020,
journal={Invent. math.},
year={2020},
volume={221},
pages={505--596},
title={Mirror symmetry for moduli spaces of {Higgs}
bundles via $p$-adic integration},
author={M. Groechenig and D. Wyss}
}

@article{PW2025,
author={T. Pierce and D. Weisbart}, title={Brownian motion in the $p$-adic integers is a
limit of discrete time random walks},
journal={J Stat Phys}, 
volume={192},
pages={104}, 
year={2025}
}

@article{Zuniga2020,
author={W.A. Zúñiga-Galindo},
title={Reaction-diffusion equations on complex networks and {Turing} patterns
via $p$-adic analysis},
journal={Journal of Mathematical Analysis and Applications}, 
volume={491},
number={1},
pages={124239}, 
year={2020}
}

@article{ExpliCycMumf,
author={P.E. Bradley},
title={Cyclic coverings of the $p$-adic projective line by {Mumford} curves},
journal={manuscripta mathematica},
volume={124},
number={1},
pages={77--95},
year={2007}
}

@article{RETMumf,
author={P.E. Bradley},
title={Riemann existence theorems of {Mumford} type},
journal={Mathematische Zeitschrift}, volume={251},
number={2},
pages={393--414}, 
year={2005}
}

@article{Weisbart2024,
author={D. Weisbart},
title={$p$-adic {Brownian} motion is a scaling limit},
journal={J. Phys. A: Math.
Theor.}, 
volume={57},
pages={205203},
year={2024}
}

@article{BW2019,
journal={Commun. Math. Phys.},
volume={369},
pages={371--402},
year={2019},
title={$p$-Adic {Brownian} Motion as a Limit of Discrete Time
Random Walks},
author={E. Bakken and D.  Weisbart}
}

@misc{stacks,
author={{The Stacks Project Authors}},
title={Stacks project},
howpublished={\url{https://stacks.math.columbia.edu/}}
}

@article{Oesterle1984,
author={J. Oesterlé},
title={Nombres de {Tamagawa} et groupes unipotents en characteristique $p$},
journal={Inventiones mathematicae}, volume={78},
pages={13--88}, 
year={1984}
}

@article{HJ2026,
journal={Physical Review D},
volume={113}, 
pages={126021},
year={2026},
title={Glimpse into the ultrametric spectrum},
author={A. Huang and C.B. Jepsen}
}

@article{HS2022,
author={M. Hinz and M. Schwarz},
title={A note on {Neumann} problems on graphs},
journal={Positivity}, 
volume={26},
pages={68}, 
year={2022}
}

@book{Vakil2017,
title={The Rising Sea: Foundations of Algebraic Geometry},
publisher={math216.wordpress.com},
year={November 18, 2017 draft},
author={R. Vakil}
}

@article{Yasuda2017,
author={T. Yasuda},
title={The wild {McKay} correspondence and $p$-adic measures},
journal={J. Eur. Math. Soc.}, volume={19},
number={12},
pages={3709--3743}, 
year={2017}
}

@article{BKL2026,
author={P. Bürgisser and A. Kulkarni and A. Lerario},
title={Nonarchimedean integral
geometry},
journal={Selecta Mathematica}, 
volume={32},
number={10}, 
year={2026}
}

\end{document}